\documentclass[preprint,review,10pt]{amsart}
\usepackage{amsmath,amssymb,amsfonts,xspace}
\newtheorem{theorem}{Theorem}[section]

\theoremstyle{definition}
\newtheorem{definition}[theorem]{Definition}
\newtheorem{example}[theorem]{Example}

\newtheorem{corollary}[theorem]{Corollary}

\theoremstyle{remark}

\numberwithin{equation}{section}

\begin{document}

\title{ some  hyperideals defined on the basis of the intersection of all $n$-ary prime hyperideals
 }

\author{M. Anbarloei}
\address{Department of Mathematics, Faculty of Sciences,
Imam Khomeini International University, Qazvin, Iran.
}

\email{m.anbarloei@sci.ikiu.ac.ir }





\begin{abstract} 
Various classes of hyperideals have been studied in many papers in order to let us fully understand the structures of hyperrings in general. The purpose of this paper is the study of some  hyperideals whose concept is created on the basis of the intersection of all  $n$-ary prime hyperideals in category of Krasner $(m, n)$-hyperrings. In this regard we introduce notions of $N$-hyperideals, $\delta$-$N$-hyperideals, $(k,n)$-absorbing $\delta$-$N$-hyperideals and $S$-$N$-hyperideals. The overall
framework of these structures is then explained, providing a number of major conclusions.
\end{abstract}
\maketitle
\section{Introduction}
The notion of prime ideal as an extension 
of the notion of prime number in the ring of integers  plays a highly important role in the theory of rings. In 2017, Terkir  et al. \cite{tekir} proposed a new classe of ideals called $n$-ideals and  investigated some properties of them analogous with prime ideals. Let  $I$ be a ideal of ring $R$. $I$ refers to an $n$-ideal  if the condition $xy \in I$ with $x \notin \sqrt{0}$ implies $y \in I$ for all $x,y \in I$.
Afterward, the concept of $(2,n)$-ideals in a commutative ring was defined  by Tamekhante and Bouba \cite{Tamekkante}. They give many results to show the relations between this new notion and others that already exist. Indeed, the notion is a generalization of $n$-ideals. A proper ideal $I$ of $R$ is said to be a $(2,n)$-ideal if  $xyz \in I$ for $x,y,z \in R$  implies that $xy \in I$ or $yz \in \sqrt{0}$ or $xz \in \sqrt{0}$.

Algebraic hyperstructures arise as natural extensions of classical algebraic structures when the composition operator is multivalued. The pioneer of this theory was the French mathematician F. Marty \cite{s1}, who introduced the concept of hypergroup in 1934 on the occasion of the 8th Congress of Scandinavian Mathematicians in Stockholm. Since the 1970s, hyperstructure theory has experienced a surge of interest , when its research field was greatly extended by the introduction of other useful notions. Nowadays, hypercompositional algebra has a variety of relationships with other areas of mathematics.
One significant type of hyperrings was introduced by Krasner, when the addition is a hyperoperation, while the multiplication is an ordinary binary operation, which is said to be Krasner hyperring. An extension  of the Krasner hyperrings, which is a subclass of $(m,n)$-hyperrings, was presented  by Mirvakili and Davvaz which is called Krasner $(m,n)$-hyperring,  in \cite{d1}. Ameri and Norouzi  defined some substancial classes of hyperideals in  Krasner $(m, n)$-hyperrings  \cite{sorc1}. Later, in \cite{rev2} the concepts of $(k,n)$-absorbing hyperideals and $(k,n)$-absorbing primary hyperideals were studied by Hila et al..
Norouzi and his colleagues illustrated a new defnition for normal hyperideals in Krasner $(m,n)$-hyperrings, with respect to that one given in \cite{d1} and they  proposed in \cite{nour} that these hyperideals correspond to strongly regular relations . 
In \cite{asadi}, Asadi and Ameri showed direct limit of a direct system in the category of Krasner $(m,n)$-hyperrigs.
Ozel Ay et al. presented the idea of $\delta$-primary on Krasner hyperrings \cite{bmb4}. The notion of $\delta$-primary hyperideals in Krasner $(m,n)$-hyperrings, which unifies the prime and primary hyperideals under one frame, was defined in \cite{mah3}. For a Krasner hyperring, the concept of $n$-hyperideals was briefly introduced by Omidi et al in \cite{Omidi}.

For a commutative Krasner $(m,n)$-hyperring $H$,   $\sqrt{0}^{(m,n)}$ is the intersection is taken over all $n$-ary prime hyperideals.    
\[\sqrt{0}^{(m,n)}= \biggm{\{} a \in H \ \vert \ \biggm{\{}
\begin{array}{lr}
g(a^{(s)},1_H^{(n-s)})= 0,& s \leq n\\
g_{(l)}(a^{(s)})=0 & s>n, s=l(n-1)+1
\end{array}
\biggm{\}} \biggm{\}}\]
is an alternative definition of $\sqrt{0}^{(m,n)}$.
In this paper, we aim to analyze some notions of hyperideals established on basis of the intersection of all $n$-ary prime hyperideals in a commutative Krasner $(m,n)$-hyperring. The paper is orgnized as follows. In Section 2, we have given some basic definitions and results of commutative Krasner $(m,n)$-hyperrings which we need to develop our paper. In Section 3, we introduce the idea of $n$-ary $N$-hyperiudeals in  a commutative Krasner $(m,n)$-hyperring and give several characterizations of them. In Section 4, we extend the notion of n-ary $N$-hyperideals to $n$-ary $\delta$-$N$-hyperideals. We obtain many specific results explaining the structures. 
Afterward, in Section 5, we study an expansion of the previous concept called $(k,n)$-absorbing $\delta$-$N$-hyperideals. Some properties of them are provided. The last section is devoted for introducing the $n$-ary $S$-$N$-hyperideals.
\section{Preliminaries}
For   a non-empty set $H$, the mapping $f : H^n \longrightarrow P^*(H)$
is called an $n$-ary hyperoperation where $P^*(H)$ is the
family of all the nonempty subsets of $H$. An $n$-ary hypergroupoid is an algebraic system $(H, f)$. Let   $H_1,..., H_n$  be non-empty subsets of $H$. Then $f(H^n_1) = f(H_1,..., H_n)$ is defined by $ \bigcup \{f(a^n_1) \ \vert \ a_i \in H_i, 1 \leq i \leq n \}$.
The sequence $a_i, a_{i+1},..., a_j$ is denoted by $a^j_i$ and so we  write \[f(a_1,..., a_i, b_{i+1},..., b_j, c_{j+1},..., c_n)=f(a^i_1, b^j_{i+1}, c^n_{j+1}).\] 
Moreover, if $b_{i+1} =... = b_j = b$, then we write $f(a^i_1, b^{(j-i)}, c^n_{j+1})$. $a^j_i$ is the empty symbol if $j< i$.  
Let $f$ be an $n$-ary hyperoperation. Then $(l(n- 1) + 1)$-ary hyperoperation $f_{(l)}$ is given by 
$f_{(l)}(a_1^{l(n-1)+1}) = f(f(..., f(f(a^n _1), a_{n+1}^{2n -1}),...), a_{(l-1)(n-1)+1}^{l(n-1)+1})$. 
Recall from \cite{d1} that an algebraic hyperstructure $(H, f, g)$ , or simply $H$, is called a commutative Krasner $(m, n)$-hyperring if
\begin{itemize} 
\item[\rm(1)]~ $(H, f$) is a canonical $m$-ary hypergroup, i.e., for every $1 \leq i < j \leq n$ and all $a_1^{2n-1} \in H$, 
$f(a^{i-1}_1, f(a_i^{n+i-1}), a^{2n-1}_{n+i}) = f(a^{j-1}_1, f(a_j^{n+j-1}), a_{n+j}^{2n-1})$ and $^{-1}:H \longrightarrow H$ is a unitary operation such that
\begin{itemize}
\item[\rm(i)]~ there exists a unique $e \in H$, such that $f(a, e^{(m-1)}) = a$ for every $a \in H$ and $e^{-1}=e$;
\item[\rm(ii)]~ for each $a \in H$ there exists a unique $a^{-1} \in H$ with $e \in f(a, a^{-1}, e^{(m-2)})$;
\item[\rm(iii)]~ if $a \in f(a^m _1)$, then   $a_i \in f(a, a^{-1},..., a^{-1}_{ i-1}, a^{-1}_ {i+1},..., a^{-1}_ m)$ for all $i$,
\item[\rm(iv)]~ $f(a_1^m)=f(a_{\sigma(1)}^{\sigma(m)})$ for all $\sigma \in \mathbb{S}_n$ and for all $a_1^m \in H$,
\end{itemize}
\item[\rm(2)]~ $(H, g)$ is a $n$-ary semigroup,
\item[\rm(3)] ~ $g(a^{i-1}_1, f(x^m _1 ), a^n _{i+1}) = f(g(a^{i-1}_1, x_1, a^n_{ i+1}),..., g(a^{i-1}_1, x_m, a^n_{ i+1}))$ for every $a^{i-1}_1 , a^n_{ i+1}, x^m_ 1 \in H$, and $1 \leq i \leq n$;
\item[\rm(4)]~ $g(0, a^n _2) = g(a_2, 0, a^n _3) = ... = g(a^n_ 2, 0) = 0$ for every $a^n_ 2 \in H$,
\item[\rm(5)]~ $g(a_1^n)=g(a_{\sigma(1)}^{\sigma(n)})$ for all $\sigma \in \mathbb{S}_n$ and for all $a_1^n \in H$.
\end{itemize}
For a non-empty subset $S$ of $H$, if $(S, f, g)$ is a Krasner $(m, n)$-hyperrin, then $S$ is called a subhyperring of $H$ . Let
$I$ be a non-empty subset of $H$.  $I$ refers to a hyperideal of $(R, f, g)$ if $(I, f)$ is an $m$-ary subhypergroup of $(H, f)$ and for every $x^n _1 \in R$ and $1 \leq i \leq n$, $g(a^{i-1}_1, I, a_{i+1}^n) \subseteq I$.
\begin{definition} \cite{d1}
Suppose that $(H_1, f_1, g_1)$ and $(H_2, f_2, g_2)$ are two Krasner $(m, n)$-hyperrings. A mapping
$d : H_1 \longrightarrow H_2$ is called a homomorphism if  
\[d(f_1(a_1,..., a_m)) = f_2(d(a_1),...,d(a_m))\]
\[d(g_1(b_1,..., b_n)) = g_2(d(b_1),...,d(b_n)) \]
for all $a^m _1 \in H_1$ and $b^n_ 1 \in H_1$.
\end{definition}
Now, We recall some definitions from \cite{sorc1}. Let $I$ be a hyperideal in a commutative Krasner $(m, n)$-hyperring $R$ with
scalar identity. ${\sqrt I}^{(m,n)}$
is the intersection is taken over all prime hyperideals of $H$ which contain $I$. If the set of all prime hyperideals containing $I$ is empty, then ${\sqrt I}^{(m,n)}=H$. Theorem 4.23 in \cite{sorc1} shows that if $a \in {\sqrt I}^{(m,n)}$ then there exists $t \in \mathbb {N}$ such that $g(a^ {(t)} , 1_H^{(n-t)} ) \in I$ for $t \leq n$, or $g_{(l)} (a^ {(t)} ) \in I$ for $t = l(n-1) + 1$.
\begin{definition} 
Let  $I$ be a proper hyperideal of a commutative Krasner $(m, n)$-hyperring $H$ with the scalar identity $1_R$. $I$ refers to  a 
\begin{itemize}
\item[\rm(1)]~ maximal hyperideal of $H$ if for every hyperideal $I^{\prime}$ of $H$, $I \subseteq I^{\prime} \subseteq H$ implies that $I^{\prime}=M$ or $I^{\prime}=H$. The intersection of all maximal hyperideals of $H$ is called the Jacobson radical of  $H$ which is  denoted by $J_{(m,n)}(H)$. If $H$ does not have any maximal hyperideal, we define $J_{(m,n)}(H)=H$.
\item[\rm(2)]~ $n$-ary prime hyperideal if for hyperideals $I_1,..., I_n$ of $H$, $g(I_1^ n) \subseteq I$ implies that $I_i \subseteq I$ for some $1 \leq i \leq n$. Lemma 4.5 in \cite{sorc1} shows that if for all $a^n_ 1 \in H$, $g(a^n_ 1) \in I$ implies that $a_i \in I$ for some $1 \leq i \leq n$, then $I$ is a prime hyperideal. 
\item[\rm(2)]~ $n$-ary primary hyperideal if $g(a^n _1) \in I$ and $a_i \notin I$, then $g(a_1^{i-1}, 1_H, a_{ i+1}^n) \in {\sqrt I}^{(m,n)}$ for some $1 \leq i \leq n$. Theorem 4.28 in \cite{sorc1} shows that if $I$ is a primary hyperideal of $H$, then ${\sqrt I}^{(m,n)}$ is prime. 
\end{itemize}
\end{definition}
\begin{definition}  Let $a$ be an element in a commutative Krasner $(m,n)$-hyperring $H$. The hyperideal generated by $a$ is denoted by $<a>$ and defined by $g(H,a,1^{(n-2)})=\{g(r,a,1^{(n-2)}) \ \vert \ r \in H\}$
\end{definition}

\begin{definition}
Let $a$ be an element in a commutative Krasner $(m,n)$-hyperring $H$. It is invertible if there exists $b \in R$ with $1_H=g(a,b,1_H^{(n-2)})$. Let $U$ be a subset  of $H$. Then $U$ is invertible if and only if every element of $U$ is invertible.
\end{definition}






\section{$n$-ary $N$-hyperideals}
 We begin the section with the following definition.
\begin{definition} 
Let $I$ be a  proper hyperideal  of a commutative Krasner $(m,n)$-hyperring $H$. $I$  refers to an $n$-ary $N$-hyperideal if whenever $x_1^n \in H$ with $g(x_1^n) \in I$ and $x_i \notin \sqrt{0}^{(m,n)}$ for some $1 \leq i \leq n$ imply that $g(x_1^{i-1},1_H,x_{i+1}^n) \in I$.
\end{definition}
\begin{example}
Let $x \geq 1$. Consider $n$-ary hyperintegral domain $([x,\infty) \cup \{0\},\boxplus,\cdot)$  where $\boxplus$ is  defined by
\[
a \boxplus b=
\begin{cases}
b \boxplus a=\{a\} & \text{if $b=0$},\\
\{min\{a,b\}\} & \text{if $a \neq 0, b \neq 0$ and $a \neq b$,}\\
[a,\infty) \cup \{0\} & \text{if $a=b\neq 0$.}
\end{cases}\]
and $"\cdot"$ is the usual multiplication. The hyperideal  $0$  is the only $n$-ary $N$-hyperideal of $([x,\infty) \cup \{0\},\boxplus,\cdot)$.
\end{example}  



\begin{theorem} \label{11}
For  an $n$-ary $N$-hyperideal $I$ of a commutative Krasner $(m,n)$-hyperring $H$, $I \subseteq \sqrt{0}^{(m,n)}$.
\end{theorem}
\begin{proof}
Let $I$ be an $n$-ary $N$-hyperideal of a commutative Krasner $(m,n)$-hyperring $H$ however $I \nsubseteq \sqrt{0}^{(m,n)}$. Assume that $x \in I$ but $x \notin \sqrt{0}^{(m,n)}$. Since  $I$ is an $n$-ary $N$-hyperideal of $H$ and $g(x,1_H^{(n-1)}) \in I$, we get $g(1_H^{(n)}) \in I$, yielding a contradiction.
 Consequently, $I \subseteq \sqrt{0}^{(m,n)}$.
\end{proof}
The next Theorem gives  a characterization of $n$-ary $N$-hyperideals.
\begin{theorem} \label{12}
Let $I$ be an $n$-ary prime hyperideal  of a commutative Krasner $(m,n)$-hyperring $H$. Then $I$ is an $n$-ary $N$-hyperideal if and only if $I =\sqrt{0}^{(m,n)}$.
\end{theorem}
\begin{proof}
$\Longrightarrow$ Let $I$ be an $n$-ary prime hyperideal  of $H$. Clearly,  $\sqrt{0}^{(m,n)} \subseteq I$. Suppose that $I$ is an $n$-ary $N$-hyperideal of $H$. Hence we obtain $I \subseteq \sqrt{0}^{(m,n)}$, by Theorem \ref{11}.  Then we conclude that $I = \sqrt{0}^{(m,n)}$. \\
$\Longleftarrow$ Let $I = \sqrt{0}^{(m,n)}$. We presume $x_1^n \in H$ with $g(x_1^n) \in I$ such that $x_i \notin \sqrt{0}^{(m,n)}$ for some $1 \leq i \leq n$. Since   $I$ is an $n$-ary prime hyperideal  of $H$ and $x_i \notin I$, then there exists $1 \leq j \leq i-1$ or $i+1 \leq  j \leq n$ such that $x_j \in I$. This implies that $g(x_1^{i-1},1_H,x_{i+1}^n) \in I$. Thus, $I$ is an $n$-ary $N$-hyperideal of $H$.
\end{proof}
In view of Theorem \ref{12}, we have the following result.
\begin{corollary} \label{13}
Let $H$ be a commutative Krasner $(m,n)$-hyperring. Then $\sqrt{0}^{(m,n)}$ is an $n$-ary prime hyperideal of $H$ if and only if it is an $n$-ary $N$-hyperideal of $H$.
\end{corollary}
\begin{proof}
$\Longrightarrow$ Suppose that $\sqrt{0}^{(m,n)}$ is an $n$-ary prime hyperideal of $H$.  By Theorem \ref{12}, $\sqrt{0}^{(m,n)}$ is an $n$-ary $N$-hyperideal of $H$.\\
$\Longleftarrow$ Let $\sqrt{0}^{(m,n)}$ be an $n$-ary $N$-hyperideal of $H$. Assume that $g(x_1^n) \in \sqrt{0}^{(m,n)}$ for some $x_1^n \in H$ such that $x_i \notin \sqrt{0}^{(m,n)}$ for some $1 \leq i \leq n$. Then we get  $g(x_1^{i-1},1_H,x_{i+1}^n) \in \sqrt{0}^{(m,n)}$. Since $\sqrt{0}^{(m,n)}$ is  an $n$-ary $N$-hyperideal of $H$, we can continue the process and obtain $x_j \in \sqrt{0}^{(m,n)}$ for some $1 \leq j \leq n$. This means that $\sqrt{0}^{(m,n)}$ is an $n$-ary prime hyperideal of $H$.
\end{proof}
\begin{theorem}
Let $I$ be a proper hyperideal of a commutative Krasner $(m,n)$-hyperring $H$. If every proper principal hyperideal is an $n$-ary $N$-hyperideal, then so is $I$.
\end{theorem}
\begin{proof}
Assume that $I$ is a proper hyperideal of $H$. Let $g(x_1^n) \in I$ for some $x_1^n \in H$ such that $x_i \notin \sqrt{0}^{(m,n)}$ for some $1 \leq i \leq n$. Since every proper principal hyperideal is an $n$-ary $N$-hyperideal and $g(x_1^n) \in \langle g(x_1^n) \rangle$, we get $g(x_1^{i-1},1_H,x_{i+1}^n) \in \langle g(x_1^n) \rangle \subseteq I$. This means that $I$ is a an $n$-ary $N$-hyperideal of $H$.
\end{proof}
\begin{theorem} \label{14}
Let $I$ be a proper hyperideal of a commutative Krasner $(m,n)$-hyperring $H$. Then the following statements are equivalent:
\begin{itemize} 
\item[\rm(1)]~ $I$ is an $n$-ary $N$-hyperideal of $H$.

\item[\rm(2)]~ $I=E_x$ where $E_x=\{y \in H \ \vert \ g(x,y,1_H^{(n-2)}) \in I\}$ for every $x \notin \sqrt{0}^{(m,n)}$.

\item[\rm(3)]~ $g(I_1^n) \subseteq I$ for some hyperideals $I_1^n$  of $H$ such that $I_i \cap (H-\sqrt{0}^{(m,n)}) \neq \varnothing$ for some $1 \leq i \leq n$ imply that $g(I_1^{i-1},1_H,I_{i+1}^n) \subseteq I$.
\end{itemize}
\end{theorem}
\begin{proof}
$(1) \Longrightarrow (2)$ Let $I$ be an $n$-ary $N$-hyperideal of $H$. We always have $I \subseteq E_x$ for every $x \in H$.  Assume that $y \in E_x$ but $x \notin \sqrt{0}^{(m,n)}$. This implies that $g(x,y,1_H^{(n-2)}) \in I$. Since $I$ is an $n$-ary $N$-hyperideal of $H$ and $x \notin \sqrt{0}^{(m,n)}$, then $y=g(y,1_H^{(n-2)}) \in I$. Therefore, we have $I=E_x$.

$(2) \Longrightarrow (3)$ Let $g(I_1^n) \subseteq I$ for some hyperideals $I_1^n$  of $H$ such that $I_i \cap (H-\sqrt{0}^{(m,n)}) \neq \varnothing$ for some $1 \leq i \leq n$. Then we get  $x_i \in I_i$ such that $x_i \notin \sqrt{0}^{(m,n)}$. Hence, $g(I_1^{i-1},x_i, I_{I+1}^n) \subseteq I$ which implies $g(I_1^{i-1},1_H,I_{i+1}^n) \subseteq E_{x_i}$. Since $I=E_{x_i}$ for every $x_i \notin \sqrt{0}^{(m,n)}$, then we get $g(I_1^{i-1},1_H,I_{i+1}^n) \subseteq I$.

$(3) \Longrightarrow (1)$ Assume that  $g(x_1^n) \in I$ for some $x_1^n \in H$ such that $x_i \notin \sqrt{0}^{(m,n)}$ for some $1 \leq i \leq n$. We get $g(\langle x_1 \rangle,..., \langle x_n \rangle) \subseteq  \langle g(x_1^n) \rangle \subseteq I$ and $\langle x_i \rangle \nsubseteq \sqrt{0}^{(m,n)}$. Then we conclude that  $g(\langle x_1 \rangle,..., \langle x_{i-1} \rangle,1_H, \langle x_{i+1} \rangle,...,\langle x_n \rangle)\subseteq I$ which means $g(x_1^{i-1},1_H,x_{i+1}^n) \in I$. Thus, $I$ is an $n$-ary $N$-hyperideal of $H$.
\end{proof}
\begin{theorem}\label{15}
Let  $T$ be a  non-empty subset  of a commutative Krasner $(m,n)$-hyperring $H$. If $I$ is an $n$-ary $N$-hyperideal of $H$ such that $T \nsubseteq I$, then $E_T=\{x \in H \ \vert \ g(x,T,1_H^{(n-2)}) \subseteq I \}$ is an $n$-ary $N$-hyperideal of $H$.
\end{theorem}
\begin{proof}
Clealy, $E_T \neq H$. Assume that $g(x_1^n) \in E_T$ for some $x_1^n \in H$ such that $x_i \notin \sqrt{0}^{(m,n)}$ for some $1 \leq i \leq n$. This means that $g(g(x_1^n),t,1_H^{(n-2)}) \in I$ for each $t \in T$ and so $g(x_i,g(x_1^{i-1},t,x_{i+1}^n),1_H^{(n-2)}) \in I$. Since $I$ is an $n$-ary $N$-hyperideal of $H$ and $x_i \notin \sqrt{0}^{(m,n)}$, we get $g(g(x_1^{i-1},1_H,x_{i+1}^n),t,1_H^{(n-2)})=g(x_1^{i-1},t,x_{i+1}^n) \in I$ which means $g(x_1^{i-1},1_H,x_{i+1}^n) \in E_T$. Consequently, $E_T$ is an $n$-ary $N$-hyperideal of $H$.
\end{proof}
\begin{theorem}\label{16}
For a maximal $N$-hyperideal $I$ of a commutative Krasner $(m,n)$-hyperring $H$, $I=\sqrt{0}^{(m,n)}$. 
\end{theorem}
\begin{proof}
Suppose that $I$ is a maximal $N$-hyperideal of $H$. Let us consider $g(x_1^n) \in I$ for some $x_1^n \in H$.  We may assume that $x_n \notin I$. By Theorem \ref{15}, we conclude that $E_{x_n}$ is an $n$-ary $N$-hyperideal of $H$ with $I \subseteq E_{x_n}$. Since  $I$ is a maximal $N$-hyperideal of $H$, we have $E_{x_n}=I$ and so $g(x_1^{n-1},1_H) \in I$. Now, we assume that $x_{n-1} \notin I$.  In a similar way, we get $g(x_1^{n-2},1_H^{(2)}) \in I$. By continuing the argument, we obtain $x_1 \in I$ which implies $I$ is an $n$-ary prime hyperideal of $H$. Then, we conclude that $I=\sqrt{0}^{(m,n)}$, by Theorem \ref{12}.
\end{proof}
The next result characterizes hyperrings admitting $n$-ary $N$-hyperideals.
\begin{theorem} \label{17}
Let $H$ be a commutative Krasner $(m,n)$-hyperring. Then $H$ admits an $n$-ary $N$-hyperideal if and only if $\sqrt{0}^{(m,n)}$ is an $n$-ary prime hyperideal of $H$. 
\end{theorem}
\begin{proof}
$\Longrightarrow$ Let $I$ be an $n$-ary $N$-hyperideal of $H$ and let $\Sigma$ be the set of all $n$-ary $N$-hyperideals of $H$. Then $\Sigma \neq \varnothing$, since $I \in \Sigma$. So $\Sigma$ is a partially ordered set with respect to set inclusion relation. Now, we take the chain $I_1 \subseteq I_2 \subseteq \cdots \subseteq I_n \subseteq \cdots$ in $ \Sigma$. Put $P=\bigcup_{i=1}^\infty I_i$. We prove $P$ is an $n$-ary $N$-hyperideal of $H$. Assume that $g(x_1^n) \in P$ for some $x_1^n \in H$ such that $x_i \notin \sqrt{0}^{(m,n)}$. Then there exists $s \in \mathbb{N}$ such that $g(x_1^n) \in I_s$. Then we get $g(x_1^{i-1},1_H,x_{i+1}^n) \in I_s \subseteq P$, as $I_s$ is an $n$-ary $N$-hyperideal of $H$ and $x_i \notin \sqrt{0}^{(m,n)}$. This means that $P$ is a upper bound of the mentioned chain.  By Zorn’s lemma, there is a hyperideal $Q$ which is maximal in $\Sigma$. Hence $Q=\sqrt{0}^{(m,n)}$ is an $n$-ary prime hyperideal of $H$, by Theorem \ref{16}.\\
$\Longleftarrow$ Let $\sqrt{0}^{(m,n)}$ be an $n$-ary prime hyperideal of $H$. Then it is an $n$-ary $N$-hyperideal of $H$, by Corollary \ref{13}.
\end{proof}
\begin{corollary}\label{18}
Let  $I$ be a  hyperideal  of a commutative Krasner $(m,n)$-hyperring $H$ with $I \subseteq \sqrt{0}^{(m,n)}$. Then $I$ is an $n$-ary $N$-hyperideal of $H$ if and only if $I$ is an $n$-ary primary hyperideal of $H$. 
\end{corollary}
\begin{proof}
Straightforward.
\end{proof}
Now, we are interested in the hyperrings over which $0$ is the only $n$-ary $N$-hyperideal of $H$.
\begin{theorem}\label{19}
Let $H$ be a commutative Krasner $(m,n)$-hyperring. Then $0$ is the only $n$-ary $N$-hyperideal of $H$ if and only if $H$ is an $n$-ary hyperintegral domain.
\end{theorem}
\begin{proof}
$\Longrightarrow$ Let $0$ be  the only $n$-ary $N$-hyperideal of $H$. Then by Theorem \ref{17}, $\sqrt{0}^{(m,n)}$ is an $n$-ary prime hyperideal of $H$. Therefore $\sqrt{0}^{(m,n)}$ is an $n$-ary $N$-hyperideal of $H$, by Corollary \ref{13}. Thus $0=\sqrt{0}^{(m,n)}$ is an $n$-ary prime hyperideal of $H$ which means $H$ is an $n$-ary hyperintegral domain.\\
$\Longleftarrow$ Let $H$ be an $n$-ary hyperintegral domain. Suppose that $I$ is an $n$-ary $N$-hyperideal of $H$. Then $I \subseteq \sqrt{0}^{(m,n)}$, by Theorem \ref{11}. Since $H$ is an $n$-ary hyperintegral domain, we get $0=\sqrt{0}^{(m,n)}$. Hence $I=0$ is the only $n$-ary $N$-hyperideal of $H$.
\end{proof}

Let $(H_1, f_1, g_1)$ and $(H_2, f_2, g_2)$ be two commutative Krasner $(m,n)$-hyperrings such that $1_{H_1}$ and $1_{H_2}$ be scalar identitis of $H_1$ and $H_2$, respectively. Then 
the $(m, n)$-hyperring $(H_1 \times H_2, f_1\times f_2 ,g_1 \times g_2 )$ is defined by $m$-ary hyperoperation
$f=f_1\times f_2 $ and n-ary operation $g=g_1 \times g_2$, as follows:
\[f_1 \times f_2((x_{1}, y_{1}),\cdots,(x_m,y_m)) = \{(x,y) \ \vert \ \ x \in f_1(x_1^m), y \in f_2(y_1^m) \}\]
$\hspace{1.2cm}
g_1 \times g_2 ((a_1,b_1),\cdots,(a_n,b_n)) =(g_1(a_1^n),g_2(b_1^n)) $,\\
for all $x_1^m,a_1^n \in H_1$ and $y_1^m,b_1^n \in H_2$ \cite{mah2}. 
\begin{theorem}\label{111}
Let $(H_1, f_1, g_1)$ and $(H_2, f_2, g_2)$ be two commutative Krasner $(m,n)$-hyperrings such that $1_{H_1}$ and $1_{H_2}$ be scalar identitis of $H_1$ and $H_2$, respectively. Then $H_1 \times H_2$ has no $n$-ary $N$-hyperideals.
\end{theorem}
\begin{proof}
Let $I_1$ and $I_2$ be hyperideals of $H_1$ and $H_2$, respectively, and let   $I_1 \times I_2$ be an $n$-ary $N$-hyperideal of  $H_1 \times H_2$.  It is clear that $(1_{H_1},0_{H_2}),(0_{H_1},1_{H_2}),(1_{H_1},1_{H_2}) \notin \sqrt{0_{H_1 \times H_2}}^{(m,n)}$
  but $g_1 \times g_2((1_{H_1},0_{H_2}),(0_{H_1},1_{H_2}),(1_{H_1},1_{H_2})^{(n-2)}) \in I_1 \times I_2$. Therefore we obtain $(1_{H_1},0_{H_2})=g_1 \times g_2((1_{H_1},0_{H_2}),(1_{H_1},1_{H_2})^{(n-1)}) \in I_1 \times I_2$ and $(0_{H_1},1_{H_2})=g_1 \times g_2((1_{H_1},1_{H_2}),(0_{H_1},1_{H_2}),(1_{H_1},1_{H_2})^{(n-2)}) \in I_1 \times I_2$. This implies that $(1_{H_1},1_{H_2})=f_1 \times f_2((1_{H_1},0_{H_2}),(0_{H_1},1_{H_2}),(0_{H_1},0_{H_2})^{(m-2)})) \in I_1 \times I_2$ which means $I_1 \times I_2 =H_1 \times H_2$.
\end{proof}
\begin{theorem} \label{112}
Suppose that $(H_1,f_1,g_1)$ and $(H_2,f_2,g_2)$ are two commutative Krasner $(m,n)$-hyperrings and $ h:H_1 \longrightarrow H_2$ is a homomorphism. Then:

$(1)$ If $h$ is a monomorphism and $I_2$ is an $n$-ary $N$-hyperideal of $H_2$, then $h^{-1} (I_2)$ is an $n$-ary $N$-hyperideal of $H_1$.

$(2)$ Let $h$ be an epimorphism and $I_1$ be a hyperideal of $H_1$ with $Ker(h) \subseteq I_1$. If $I_1$ is an $n$-ary $N$-hyperideal of $H_1$, then $h(I_1)$ is an $n$-ary $N$-hyperideal of $H_2$.
\end{theorem}
\begin{proof}
$(1)$ Let $g_1(x_1^n) \in h^{-1} (I_2)$ for $x_1^n \in H_1$ such that $x_i \notin \sqrt{0_{H_1}}^{(m,n)}$ for some $1 \leq i \leq n$. Then we have $h(g_1(x_1^n))=g_2(h(x_1^n)) \in I_2$. Since $h$ is a monomorphism and $x_i \notin \sqrt{0_{H_1}}^{(m,n)}$, then $h(x_i) \notin \sqrt{0_{H_2}}^{(m,n)}$.  Since $I_2$ is a $N$-hyperideal of $H_2$, we get the result that  
\[g_2(h(x_1),...,h(x_{i-1}),1_{H_2},h(x_{i+1}),...,h(x_n))\]

$\hspace{3cm}=
h(g_1(x_1^{i-1},1_{H_1},x_{i+1}^n))$

$ \hspace{3cm}\in I_2$\\
which follows $g_1(x_1^{i-1},1_{H_1},x_{i+1}^n) \in h^{-1}(I_2)$. Thus $h^{-1}(I_2)$ is an $n$-ary $N$-hyperideal of $H_1$.

$(2)$ Let  $g_2(y_1^n) \in h(I_1)$ for some  $y_1^n \in H_2$ such that $y_i \notin \sqrt{0_{H_2}}^{(m,n)}$. Since $h$ is an epimorphism, then there exist $x_1^n \in R_1$ with $h(x_1)=y_1,...,h(x_n)=y_n$. Therefore
$h(g_1(x_1^n))=g_2(h(x_1),...,h(x_n))=g_2(y_1^n) \in h(I_1)$.
Since $Ker(h) \subseteq I_1$, then we have $g_1(x_1^n) \in I_1$. Since $y_i \notin \sqrt{0_{H_2}}^{(m,n)}$, then $x_i \notin \sqrt{0_{H_1}}^{(m,n)}$. Since $I_1$ is an n-ary  $N$-hyperideal of $R_1$ and $x_i \notin \sqrt{0_{H_1}}^{(m,n)}$, we obtain $g_1(x_1^{i-1},1_{H_1},x_{i+1}^n) \in I_1$ which means 
\[h(g_1(x_1^{i-1},1_{H_1},x_{i+1}^n))
=g_2(h(x_1),...,h(x_{i-1}),1_{H_2},h(x_{i+1}),...,h(x_n))\]
$\hspace{4.5cm}=g_2(y_1^{i-1},1_{H_2},y_{i+1}^n) \in h(I_1)$\\
Consequently,  $h(I_1)$ is an $n$-ary $N$-hyperideal of $H_2$.
\end{proof}
\begin{corollary}
Let $H^{\prime}$ be a subhyperring of a commutative Krasner $(m, n)$-hyperring $H$. If $I$ is an $n$-ary $N$-hyperideal of $H$ such that $H^{\prime} \nsubseteq I$, then $H^{\prime} \cap I$ is an $n$-ary $N$-hyperideal of $H^{\prime}.$
\end{corollary}
\begin{proof}
Let us consider the identity map $j$ from $H^{\prime}$ into $H$. We conclude that $H^{\prime} \cap I$ is an $n$-ary $N$-hyperideal of $H^{\prime}$, by Theorem \ref{112} (2).
\end{proof}
Let $J$ be a hyperideal of a commutative Krasner $(m, n)$-hyperring $(H, f, g)$. Then the set 
\[H/J = \{f(a^{i-1}_1, J, a^m_{i+1}) \ \vert \ a^{i-1}_1,a^m_{i+1} \in H \}\]
endowed with m-ary hyperoperation $f$ which for all $a_{11}^{1m},...,a_{m1}^{mm} \in H$

\[f(f(a_{11}^{1 (i-1)},J, a^{1m}_ {1(i+1)}),..., f(a_{m1}^{ m(i-1)}, J, a^{mm}_ {m(i+1)}))\] 
$\hspace{3.cm}= f (f(a^{m1}_{11}),..., f(a^{m(i-1)}_{1(i-1)}), J, f(a^{m(i+1)}_{1(i+1)} ),..., f(a^{mm}_ {1m}))$\\
and with $n$-ary hyperoperation g which for all $a_{11}^{1m},...,a_{n1}^{nm} \in H$
\[g(f(a_{11}^{1 (i-1)}, J, a^{1m}_ {1(i+1)}),..., f(a_{n1}^{ n(i-1)}, J, a^{nm}_ {n(i+1)}))\] 

$\hspace{3.cm}= f (g(a^{n1}_{11}),..., g(a^{n(i-1)}_{1(i-1)}), J, g(a^{n(i+1)}_{1(i+1)} ),..., f(a^{nm}_ {1m}))$\\
construct a commutative Krasner $(m, n)$-hyperring, and $(H/J, f, g)$ is said to be the quotient Krasner $(m, n)$-hyperring of $H$ by $J$ \cite{sorc1}.

Now, we determine when the hyperideal $I/J$ is n-ary $N$-hyperideal in $H/J$.
\begin{theorem} \label{113}
Let $I$ and $J$ be two hyperideals of a commutative Krasner $(m,n)$-hyperring $H$ with $J \subseteq I$. If $I$ is an $n$-ary $N$-hyperideal of $H$, then $I/J$ is an $n$-ary $N$-hyperideal of $H/I$. 
\end{theorem}
\begin{proof}
Consider the projection map of $H$ of $J$, that is, $\theta: H \longrightarrow H/J$, defined by $a \longrightarrow f(a,I,0^{(m-2)})$.  By using Theorem \ref{112} (2), we are done. 
\end{proof}
\begin{theorem} \label{114}
Let $I$ and $J$ be two hyperideals of a commutative Krasner $(m,n)$-hyperring $H$ with $J \subseteq I$. If $I/J$ is an $n$-ary $N$-hyperideal of $H/J$ with $J \subseteq \sqrt{0_H}^{(m,n)}$, then $I$ is an $n$-ary $N$-hyperideal of $H$. 
\end{theorem}
\begin{proof} Let  $g(x_1^n) \in I$ for $x_1^n \in H$ such that $x_i \notin \sqrt{0_H}^{(m,n)}$ for some $1 \leq i \leq n$. Then we get 
$g(f(x_1,J,0_H^{(m-2)}),\cdots,f(x_n,J,0_H^{(m-2)}))=f(g(x_1^n),J,0_{H/J}^{(m-2)}) \in I/J$. Note that $f(x_i,J,0_H^{(m-2)}) \notin \sqrt{0_{H/J}}^{(m,n)}$. Since $I/J$ is an $n$-ary $N$-hyperideal of $H/J$ and $f(x_i,J,0_H^{(m-2)}) \notin \sqrt{0_{H/J}}^{(m,n)}$, we get the result that \[g(f(x_1,J,0_H^{(m-2)}),\cdots,f(x_{i-1},J,0_H^{(m-2)}),1_{H/J},f(x_{i+1},J,0_H^{(m-2)}),\]
$\hspace{2cm} \cdots,f(x_n,J,0_H^{(m-2)})) \in I/J$\\
which implies $f(g(x_1^{i-1},1_H,x_{i+1}^n),J,0_{H/J}^{(m-2)}) \in I/J$. Therefore $g(x_1^{i-1},1_H,x_{i+1}^n) \in I$.
Thus $I$ is an n-ary $N$-hyperideal of $H$.
\end{proof}
Using Theorem \ref{11} and Theorem \ref{114}, we have the next corollary.
\begin{corollary} \label{115}
Let $I$ and $J$ be two hyperideals of a commutative Krasner $(m,n)$-hyperring $H$ with $J \subseteq I$. If $I/J$ is an $n$-ary $N$-hyperideal of $H/J$ such that $J$ is an $n$-ary $N$-hyperideal of $H$, then $I$ is an $n$-ary $N$-hyperideal of $H$. 
\end{corollary}
\begin{proof}
The proof follows from Theorem \ref{11} and Theorem \ref{114}.
\end{proof}
\section{$n$-ary $\delta$-$N$-hyperideals}
Recall from  \cite{mah3} that a function $\delta$ is a hyperideal expansion of a commutative Krasner $(m,n)$-hyperring $H$ if it assigns to each hyperideal $I$ of $H$ a hyperideal $\delta(I)$ of $H$ such that 
 $I \subseteq \delta(I)$ and 
 if $I \subseteq J$ for any hyperideals $I, J$ of $H$, then $\delta (I) \subseteq \delta (J)$. For instance, $\delta_0(I)=I$, $\delta_1(I)=\sqrt{I}^{(m,n)}$ and $\delta_H(I)=H$ for all hyperideals $I$ of $H$ are  hyperideal expansions of $H$. Moreover, $\delta_q(I/J)=\delta(I)/J$  for expansion function $\delta$ of $H$ and  for all hyperideals $I$ of $H$ containing hyperideal $J$ is an hyperideal expansion of $H$. By using a hyperideal expansion $\delta$ of $H$, we present the following definition.
\begin{definition}
Assume that  $\delta$ is a hyperideal expansion of a commutative Krasner $(m,n)$-hyperring $H$. A proper hyperideal $I$ of $H$ is said to be an $n$-ary $\delta$-$N$-hyperideal if  $g(x_1^n) \in I$ for $x_1^n \in H$ and $x_i \notin \sqrt{0}^{(m,n)}$ for some $1 \leq i \leq n$ imply that $g(x_1^{i-1},1_H,x_{i+1}^n) \in \delta(I).$
\end{definition}
\begin{example} \label{classes}
Assume that $\mathbb{Z}_{12}=\{0,1,2,3,\cdots,11\}$ is the set of all congruence classes of integers modulo $12$ and $\mathbb{Z}^{\star}_{12}=\{1,5,7,11\}$ is multiplicative subgroup of units $\mathbb{Z}_{12}$. Construct $H$ as $\mathbb{Z}_{12}/\mathbb{Z}_{12}^{\star}$. Then we have  $H=\{\bar{0},\bar{1},\bar{2},\bar{3},\bar{4},\bar{6}\}$ in which $\bar{0}=\{0\}$, $\bar{1}=\{1,5,7,11\}$, $\bar{2}=\bar{10}=\{2,10\}$, $\bar{3}=\bar{9}=\{3,9\}$,  $\bar{4}=\bar{8}=\{4,8\}$, $\bar{6}=\{6\}$. Consider Krasner hyperring $(H,\boxplus,\circ)$ that for all $\bar{x},\bar{y} \in H$, $\bar{x} \circ \bar{y}=\overline{xy}$ and  2-ary hyperoperation $\boxplus$ is defined as follows

\[\begin{tabular}{|c|c|c|c|c|c|c|} 
\hline $\boxplus$ & $\bar{0}$ & $\bar{1}$ & $\bar{2}$ & $\bar{3}$ & $\bar{4}$ & $\bar{6}$
\\ \hline  $\bar{0}$ & $\bar{0}$ & $\bar{1}$ & $\bar{2}$ & $\bar{3}$ & $\bar{4}$ & $\bar{6}$
\\ \hline $\bar{1}$ & $\bar{1}$ & $\bar{0},\bar{2},\bar{\bar{4}},\bar{6}$ & $\bar{1},\bar{3}$ & $\bar{2},\bar{4}$ & $\bar{1},\bar{3}$ & $\bar{1}$
\\ \hline $\bar{2}$ & $\bar{2}$ & $\bar{1},\bar{3}$ & $\bar{0},\bar{4}$ & $\bar{1}$ & $\bar{2},\bar{6}$ & $\bar{4}$
\\ \hline $\bar{3}$ & $\bar{3}$ & $\bar{2},\bar{4}$ & $\bar{1}$ & $\bar{0},\bar{6}$ & $\bar{1}$ & $\bar{3}$
\\ \hline $\bar{4}$ & $\bar{4}$ & $\bar{1},\bar{3}$ & $\bar{2},\bar{6}$ &  $\bar{1}$ & $\bar{0},\bar{4}$ & $\bar{2}$
\\ \hline $\bar{6}$ & $\bar{6}$ & $\bar{1}$ & $\bar{4}$ & $\bar{3}$ & $\bar{2}$ & $\bar{0}$
\\ \hline
\end{tabular}\]

In the hyperring, $\sqrt{0}^{(2,2)}=\{\bar{0},\bar{6}\}$ and $I=\{\bar{0},\bar{2},\bar{4},\bar{6}\}$ is a 2-ary $\delta_1$-$N$-hyperring.
\end{example}

\begin{theorem} \label{21}
Let $I$ be an $n$-ary $\delta$-primary hyperideal of a commutative Krasner $(m,n)$-hyperring $H$ such that $I \subseteq \sqrt{0}^{(m,n)}$. Then $I$ is an $n$-ary $\delta$-$N$-hyperideal of $H$.
\end{theorem}
\begin{proof}
Let $g(x_1^n) \in I$ for $x_1^n \in H$ such that $x_i \notin \sqrt{0}^{(m,n)}$ for some $1 \leq i \leq n$. By the assumption, $x_i \notin I$. Since $I$ is an $n$-ary $\delta$-primary hyperideal of $H$, we conclude that $g(x_1^{i-1},1_H,x_{i+1}^n) \in \delta(I)$. Thus $I$ is an $n$-ary $\delta$-$N$-hyperideal of $H$.
\end{proof}
The next Theorem shows that the inverse of Theorem \ref{21} is true if $I=\sqrt{0}^{(m,n)}$.
\begin{theorem}\label{22}
Let $\sqrt{0}^{(m,n)}$ be an $n$-ary $\delta$-$N$-hyperideal of a commutative Krasner $(m,n)$-hyperring $H$. Then $\sqrt{0}^{(m,n)}$ is an $n$-ary $\delta$-primary hyperideal of $H$.
\end{theorem}
\begin{proof}
Let $g(x_1^n) \in \sqrt{0}^{(m,n)}$ for some $x_1^n \in H$ such that $x_i \notin \sqrt{0}^{(m,n)}$ for some $1 \leq i \leq n$. Since $\sqrt{0}^{(m,n)}$ is an $n$-ary $\delta$-$N$-hyperideal of $H$, we get the result that $g(x_1^{i-1},1_H,x_{i+1}^n) \in \delta(\sqrt{0}^{(m,n)})$, as needed.
\end{proof}

\begin{theorem} \label{23}
Let $I$ be a proper hyperideal of a commutative Krasner $(m,n)$-hyperring $H$. Then the followings
are equivalent:
\begin{itemize} 
\item[\rm(1)]~ $I$ is an $n$-ary $\delta$-$N$-hyperideal of $H$.

\item[\rm(2)]~ $E_x \subseteq \sqrt{0}^{(m,n)}$ for all $x \notin \delta(I)$ where $E_x=\{y \in H \ \vert \ g(x,y,1_H^{(n-2)}) \in I\}$.

\item[\rm(3)]~ If $g(x,I_1^{n-1}) \subseteq I$ for some hyperideals $I_1^{n-1}$ of $H$ and for some $x \in H$ implies $x \in \sqrt{0}^{(m,n)}$ or $g(1_H,I_1^{n-1}) \subseteq \delta(I)$.

\item[\rm(4)]~ If $g(I_1^n) \subseteq I$ for some hyperideals $I_1^{n-1}$ of $H$, then $I_i \cap (H-\sqrt{0}^{(m,n)}) = \varnothing$ for some $1 \leq i \leq n$ or $g(I_1^{i-1},1_H,I_{i+1}^n) \subseteq \delta(I)$.
\end{itemize}
\end{theorem}
\begin{proof}
$(1) \Longrightarrow (2)$ Assume that $y \in E_x$. So $g(y,x,1_H^{(n-2)}) \in I$. Since $I$ is an $n$-ary $\delta$-$N$-hyperideal of $H$ and $x=g(x,1_H^{(n-2)}) \notin  \delta(I)$, we obtain $y \in \sqrt{0}^{(m,n)}$ which means $E_x \subseteq \sqrt{0}^{(m,n)}$.\\
$(2) \Longrightarrow (3)$ Let $g(x,I_1^{n-1}) \subseteq I$ for some hyperideals $I_1^{n-1}$ of $H$ and for some $x \in H$ such that $g(1_H,I_1^{n-1}) \nsubseteq \delta(I)$. Hence there exist $a_1 \in I_1, \cdots , a_{n-1} \in I_{n-1}$ such that $g(1_H,a_1^{n-1}) \notin \delta(I)$. Since $x \in E_{g(1_H,a_1^{n-1})}$, we conclude that $x \in \sqrt{0}^{(m,n)}$.\\
$(3) \Longrightarrow (4)$ Let $g(I_1^n) \subseteq I$ for some hyperideals $I_1^{n-1}$ of $H$ such that  $I_i \cap (H-\sqrt{0}^{(m,n)}) \neq \varnothing$ for some $1 \leq i \leq n$. Therefore we have some $x \in I_i \cap (H-\sqrt{0}^{(m,n)})$. Since $g(I_1^{i-1},x,I_{i+1}^n) \subseteq I$, we get $g(I_1^{i-1},1_H,I_{i+1}^n) \subseteq \delta(I)$, as $x \notin \sqrt{0}^{(m,n)}$.\\
$(4) \Longrightarrow (1)$ Suppose that $g(x_1^n) \in I$ for $x_1^n \in H$ such that $x_i \notin \sqrt{0}^{(m,n)}$ for some $1 \leq i \leq n$. Let us consider $I_1=\langle x_1 \rangle,\cdots, I_n=\langle x_n \rangle$. Since $I_i \cap (H-\sqrt{0}^{(m,n)}) \neq \varnothing$, we conclude that $g(x_1^{i-1},1_H,x_{i+1}^n) \in g(I_1^{i-1},1_H,I_{i+1}^n) \subseteq \delta(I)$. Thus $I$ is an $n$-ary $\delta$-$N$-hyperideal of $H$.
\end{proof}
\begin{theorem} \label{24}
Let $I$ be an $n$-ary $\delta$-$N$-hyperideal of a commutative Krasner $(m,n)$-hyperring $H$ and let $x \notin \delta(I).$ If $F_x \subseteq \delta(E_x)$ where $F_x=\{y \in H \ \vert \ g(y,x,1_H^{(n-2)}) \in \delta(I)\}$ and $E_x=\{y \in H \ \vert \ g(y,x,1_H^{(n-2)}) \in I\}$, then $E_x$ is an $n$-ary $\delta$-$N$-hyperideal of $H$.
\end{theorem}
\begin{proof}
Let $g(x_1^n) \in E_x$ for $x_1^n \in H$ such that $x_i \notin \sqrt{0}^{(m,n)}$ for some $1 \leq i \leq n$. Hence $g(g(x_1^n),x,1_H^{(n-2)})=g(x_i,g(x_1^{i-1},x,x_{i+1}^n),1_H^{(n-2)}) \in I$. Since $I$ is  an $n$-ary $\delta$-$N$-hyperideal of $H$, we have $g(x_1^{i-1},x,x_{i+1}^n) \in \delta(I)$. Hence $g(x_1^{i-1},1_H,x_{i+1}^n) \in F_x \subseteq \delta (E_x)$, as needed.
\end{proof}

Recall  that a hyperideal expansion $\delta$ of $H$ is intersection preserving  if it satisfies $\delta(I_1 \cap I_2)=\delta(I_1) \cap \delta(I_2)$ for each hyperideals $I_1$ and $ I_2$ of a commutative Krasner $(m,n)$-hyperring $H$. For instance, the hyperideal expansion $\delta_1$ of $H$ is intersection preserving.
\begin{theorem}\label{25}
Let $I_1^n$ be some $n$-ary $\delta$-$N$-hyperideals of a Krasner $(m,n)$-hyperring $H$ and let the hyperideal expansion $\delta$ of $H$ be intersection preserving. Then $\bigcap_{i=1}^n I_i$ is an $n$-ary $\delta$-$N$-hyperideal of $H$.
\end{theorem}
\begin{proof}
Put $I =\bigcap_{i=1}^n I_i$. Suppose that $g(x_1^n) \in I$ for $x_1^n \in H$ with $g(x_1^{i-1},1_H,x_{i+1}^n) \notin \delta (I)$. Since the hyperideal expansion $\delta$ of $H$ is intersection preserving, then there exists $1 \leq t \leq n$ with $g(x_1^{i-1},1_H,x_{i+1}^n) \notin \delta (I_t)$. As $I_t$ ia an $n$-ary $\delta$-$N$-hyperideal of $H$, we have $x_i \in \sqrt{0}^{(m,n)}$. Therefore $I =\bigcap_{i=1}^n I_i$ is an $n$-ary $\delta$-$N$-hyperideal of $H$.
\end{proof}
\begin{theorem}\label{26}
Suppose that $I_1,I_2$ and $I_3$ are  proper hyperideals of a commutative Krasner $(m,n)$-hyperring $H$ with $I_1 \subseteq I_2 \subseteq I_3$ and $\delta(I_1)=\delta(I_3)$. If $I_3$ is an $n$-ary $\delta$-$N$-hyperideal of $H$, then $I_2$ is an $n$-ary $\delta$-$N$-hyperideal of $H$.
\end{theorem}
\begin{proof}
Let $g(x_1^n) \in I_2$ for some $x_1^n \in H$ such that $x_i \notin \sqrt{0}^{(m,n)}$. Since $I_3$ is an $n$-ary $\delta$-$N$-hyperideal of $H$ and $I_2 \subseteq I_3$, then we conclude that  $g(x_1^{i-1},1_H,x_{i+1}^n) \in \delta(I_3)$. Then we get $g(x_1^{i-1},1_H,x_{i+1}^n) \in \delta(I_1)$. Therefore $g(x_1^{i-1},1_H,x_{i+1}^n) \in \delta(I_2)$, as $I_1 \subseteq I_2$. Thus $I_2$ is an $n$-ary $\delta$-$N$-hyperideal of $H$.
\end{proof} 
\begin{theorem} \label{27} 
Suppose that  $I$ is a proper hyperideal of a commutative Krasner $(m,n)$-hyperring $H$ such that $\delta(I)$ is an $n$-ary $N$-hyperideal of $H$. Then $I$ is an $n$-ary $\delta$-$N$-hyperideal of $H$.
\end{theorem}
\begin{proof}
Let $\delta(I)$ be an $n$-ary $N$-hyperideal of $H$. Let $g(x_1^n) \in I$ for $x_1^n \in H$ with $x_i \notin\sqrt{0}^{(m,n)}$ for some $1 \leq i \leq n$. Since $I \subseteq \delta(I)$ and $\delta(I)$ is an $n$-ary $N$-hyperideal of $H$, then $g(x_1^{i-1},1_H,x_{i+1}^n) \in \delta(I)$. It follows that $I$ is an $n$-ary $\delta$-$N$-hyperideal of $H$.
\end{proof}
The inverse of the previous theorem  is true if $\delta=\delta_1$. See the next Theorem.
\begin{theorem}\label{28}
If $I$ is an $n$-ary $\delta_1$-$N$-hyperideal of a commutative Krasner $(m,n)$-hyperring $H$, then $\delta_1(I)$ is an $n$-ary $N$-hyperideal of $H$.
\end{theorem}
\begin{proof}
Let $g(x_1^n)\in \delta_1(I)$ for $x_1^n \in H$ such that $x_i \notin\sqrt{0}^{(m,n)}$. It implies that there exists $t \in \mathbb{N}$ such that if $t \leq n$, then $g(g(x_1^n)^{(t)},1_H^{(n-t)}) \in I$. Therefore we get
\[g(x_i^{(t)},g(x_i^{i-1},1_H,x_{i+1}^n)^{(t)},1_H^{(n-2t)})\]

$\hspace{3.4cm}=g(x_i^{(t)},g(x_i^{i-1},1_H,x_{i+1}^n)^{(t)},g(1_R^{(n)}),1_H^{(n-2t-1)})$

$\hspace{3.4cm}=g(g(x_i^{(t)},1_H^{(n-t)}),g(x_i^{i-1},1_H,x_{i+1}^n)^{(t)},1_H^{(n-t-1)})$

$\hspace{3.4cm}\subseteq I$.\\
Since $I$ is an $n$-ary $\delta_1$-$N$-hyperideal of $H$ and 
$g(x_i^{(t)},1^{(n-t)}) \notin \sqrt{0}^{(m,n)}$, we get the result that
$g(g(x_i^{i-1},1_H,x_{i+1}^n)^{(t)},1_H^{(n-t)}) \in \delta_1(I)$
and so $g(x_i^{i-1},1_H,x_{i+1}^n) \in \delta_1(I)$. It follows that $\delta_1(I)$ is an $n$-ary $N$-hyperideal of $H$.
 By using a similar argument, one can easily complete the proof where $t=l(n-1)+1$.
\end{proof}
\begin{theorem}
Suppose that $\sqrt{0}^{(m,n)}$ is  the only maximal hyperideal of a commutative Krasner $(m,n)$-hyperring $H$. Then for all $a \in H$, $\langle a \rangle$ is an $n$-ary hyperideal of $H$.
\end{theorem}
\begin{proof}
Assume that $\sqrt{0}^{(m,n)}$ is the only maximal hyperideal of a commutative Krasner $(m,n)$-hyperring $H$. Let $a \in H$ and $g(a_1^n) \in \langle a \rangle$ for $a_1^n \in H$ such that $a_i \notin \langle a \rangle$ for some $1 \leq i \leq n$. From $a_i \notin \langle a \rangle$, it follows that $a_i$ is an invertible element. Then we have $g(a_1^{i-1},1_H,a_{i+1}^n) \in \langle a \rangle\subseteq \delta(\langle a \rangle)$ which means that $\langle a \rangle$ is an $n$-ary hyperideal of $H$.
\end{proof}
\begin{theorem}
Let $I$ be a hyperideal of a commutative Krasner $(m,n)$-hyperring $H$ and $\delta$ be a a hyperideal expansion of $H$ such that $\delta( \delta(I))=\delta(I)$. If $I$ is an $n$-ary $\delta$-$N$-hyperideal of $H$ and $x \notin \sqrt{0}^{(m,n)}$, then $\delta
(E_x)=\delta(I)$ where $E_x=\{y \in H \ \vert \ g(y,x,1_H^{(n-2)}) \in I\}$.
\end{theorem}
\begin{proof}
Assume that $I$ is an $n$-ary $\delta$-$N$-hyperideal of $H$ and $x \notin \sqrt{0}^{(m,n)}$. Let $y \in E_x$. This means $g(y,x,1_H^{(n-2)}) \in I$. Since  $I$ is an $n$-ary hyperideal of $H$ and $x \notin \sqrt{0}^{(m,n)}$, we get $y=(y,1^{(n-1)}) \in \delta(I)$ which implies $E_x \subseteq \delta(I)$ and so $\delta(E_x) \subseteq \delta(\delta(I))$. By the assumpption we have $\delta(E_x) \subseteq \delta(I)$. On the other hand, we have $I \subseteq E_x$ and so $\delta(I) \subseteq \delta(E_x)$. Consequently, $\delta(E_x) \subseteq \delta(I)$.
\end{proof}
\begin{theorem}
Suppose that $I$ is a maximal $n$-ary  $\delta$-$N$-hyperideal of a commutative Krasner $(m,n)$-hyperring $H$ and  $x \notin \delta(I).$ If $F_x \subseteq \delta(E_x)$ where $F_x=\{y \in H \ \vert \ g(y,x,1_H^{(n-2)}) \in \delta(I)\}$ and $E_x=\{y \in H \ \vert \ g(y,x,1_H^{(n-2)}) \in I\}$, then $I=\sqrt{0}^{(m,n)}$ is an $n$-ary prime hyperideal of $H$.
\end{theorem}
\begin{proof}
Let $I$ be a maximal $n$-ary  $\delta$-$N$-hyperideal of $H$. Assume that $g(a_1^n) \in I$ for $a_1^n \in H$  such that $a_i \notin I$ for some $1 \leq i \leq n$. Then, by Theorem \ref{24}, $E_x$ is an $n$-ary $\delta$-$N$-hyperideal of $H$. Therefore $E_x=I$ by the maximality of $I$. Thus we have $g(a_1^{i-1}1_H,a_{i+1}^n) \in I$. Since $I$ is a maximal $n$-ary  $\delta$-$N$-hyperideal of $H$, we can
continue the process and so get $a_j \in I$ for some $1 \leq j \leq n$ which implies $I$ is an $n$-ary prime hyperideal of $H$. This means that $\sqrt{0}^{(m,n)} \subseteq I$. Now, suppose that $I \nsubseteq \sqrt{0}^{(m,n)}$. Then there exist some $x \in I$ such that $x \notin \sqrt{0}^{(m,n)}$. Since $g(x,1_H^{(n-1)}) \in I$ and $I$ is an  $n$-ary  $\delta$-$N$-hyperideal of $H$, then we get the result that $g(1_H^{(n)}) \in I$ which is a contradiction. Then we conclude that $I = \sqrt{0}^{(m,n)}$.
\end{proof}
\begin{theorem}\label{29}
Let $\gamma$ and $\delta$ be two  hyperideal expansions of a commutative Krasner $(m,n)$-hyperring $H$.   If $I$ is an $n$-ary $\gamma$-$N$-hyperideal of $H$ and $\gamma(J) \subseteq \delta(J)$ for all hyperideals $J$ of $H$, then $I$ is an $n$-ary $\delta$-$N$-hyperideal of $H$. Moreover, if $\delta(I)$  is an $n$-ary $\gamma$-$N$-hyperideal of $H$, then $I$ is an $n$-ary $\gamma \circ \delta$-$N$-hyperideal of $H$. 
\end{theorem}
\begin{proof}
The proof of the first assertion is straightforward. For the second assertion, suppose that $g(x_1^n) \in I$ for $x_1^n \in H$ such that $x_i \notin\sqrt{0}^{(m,n)}$. Therefore we get  $g(x_1^n) \in \delta(I)$, as $I \subseteq \delta(I)$. By the assumption, we conclude that $g(x_1^{i-1},1_H,x_{i+1}^n) \in \gamma(\delta(I))=\gamma \circ \delta(I)$ which means $I$ is an $n$-ary $\gamma \circ \delta$-$N$-hyperideal of $H$.
\end{proof}
Assume that $(H_1,f_1,g_1)$ and $(H_2,f_2,g_2)$ are two commutative Krasner $(m,n)$-hyperrings and $h: H_1 \longrightarrow H_2$ a hyperring homomorphism. Let $\delta$ and $\gamma$ be two hyperideal expansions of $H_1$ and $H_2$, respectively. Recall from \cite{mah3} that $h$ is  a $\delta \gamma$-homomorphism if $\delta(h^{-1}(I_2)) = h^{-1}(\gamma(I_2))$ for the hyperideal $I_2$ of
$H_2$. Note that $\gamma(h(I_1)=h(\delta(I_1)$ for $\delta \gamma$-epimorphism $h$ and for hyperideal $I_1$ of $H_1$ such that $ Ker (h) \subseteq I_1$. \\For instance, suppose that $(H_1,f_1,g_1)$ and $(H_2,f_2,g_2)$ are two Krasner $(m,n)$-hyperrings. If $\delta_1$ of $H_1$ and $\gamma_1$ of $H_2$ are the hyperideal expansions defined in Example 3.2 of \cite{mah3}, then all homomorphism $h:H_1 \longrightarrow H_2$ is a $\delta_1 \gamma_1$-homomorphism.
\begin{theorem}\label{210}
Let $\delta$ and $\gamma$ be two hyperideal expansions of commutative Krasner $(m,n)$-hyperrings $(H_1,f_1,g_1)$ and $(H_2,f_2,g_2)$, respectively, and let $ h:H_1 \longrightarrow H_2$ be a $\delta \gamma$-homomorphism. Then the followings hold : 
\begin{itemize} 
\item[\rm(1)]~ If $I_2$ is an $n$-ary $\gamma$-$N$-hyperideal of $H_2$ and $h$ is a monomorphism, then $h^{-1} (I_2)$ is an $n$-ary $\delta$-$N$-hyperideal of $H_1$.
\item[\rm(2)]~ If $h $ is  an epimorphism and $I_1$ is an $n$-ary $\delta$-$N$-hyperideal of $H_1$ containing $Ker(h)$, then $h(I_1)$ is an $n$-ary $\gamma$-$N$-hyperideal of $H_2$.
\end{itemize} 
\end{theorem}
\begin{proof}
$(1)$ Let  $g_1(x_1^n) \in h^{-1} (I_2)$ for  $x_1^n \in H_1$.  It follows that $g_2(h(x_1),\cdots,h(x_n))=h(g_1(x_1^n)) \in I_2$. Since $I_2$ is an $n$-ary $\gamma$-$N$-hyperideal of $H_2$, then $h(x_i) \in \sqrt{0_{H_2}}^{(m,n)}$ for some $1 \leq i \leq n$ or 
\[g_2(h(x_1),...,h(x_{i-1}),1_{H_2},h(x_{i+1}),...,h(x_n))\]
$\hspace{3.4cm}=
h(g_1(x_1^{i-1},1_{H_1},x_{i+1}^n))$

$\hspace{3cm}\in \gamma(I_2)$.\\
In the first possibility, since $Ker(h)=\{0_{H_1}\}$, we obtain $x_i \in \sqrt{0_{H_1}}^{(m,n)}$. In the second possibility, we get the result that  $g_1(x_1^{i-1},1_{H_1},x_{i+1}^n) \in h^{-1}(\gamma(I_2))=\delta(h^{-1}(I_2)$. Consequently,  $h^{-1}(I_2)$ is a $\delta$-$N$-hyperideal of $H_1$.

$(2)$ Assume that $g_2(y_1^n) \in h(I_1)$ for $y_1^n \in H_2$ with $y_i \notin \sqrt{0_{H_2}}^{(m,n)}$ for some $1 \leq i \leq n$.  Then there exist $x_1^n \in H_1$ such that $h(x_1)=y_1,...,h(x_n)=y_n$ as $h$ is an epimorphism. Therefor
$h(g_1(x_1^n))=g_2(h(x_1),...,h(x_n))=g_2(y_1^n) \in h(I_1)$.\\
Since $ I_1$ contains $Ker(h)$, we have $g_1(x_1^n) \in I_1$. Since $y_i \notin \sqrt{0_{H_2}}^{(m,n)}$, then $x_i \notin \sqrt{0_{H_1}}^{(m,n)}$. Since $I_1$ is a $\delta$-$N$-hyperideal of $H_1$ and $x_i \notin \sqrt{0_{H_1}}^{(m,n)}$, it follows that $g_1(x_1^{i-1},1_{H_1},x_{i+1}^n) \in \delta(I_1)$ which implies 
\[h(g_1(x_1^{i-1},1_{H_1},x_{i+1}^n))
=g_2(h(x_1),...,h(x_{i-1}),1_{H_2},h(x_{i+1}),...,h(x_n))\]

$\hspace{4.1cm}=g_2(y_1^{i-1},1_{H_2},y_{i+1}^n)$

$\hspace{4.1cm}\in h(\delta(I_1)).$\\
By the assumption, we have $h(\delta(I_1))=\gamma(h(I_1))$. So $g_2(y_1^{i-1},1_{H_2},y_{i+1}^n) \in \gamma(h(I_1))$. Hence $h(I_1)$ is an $n$-ary  $\gamma$-$N$-hyperideal of $H_2$.
\end{proof}
\begin{corollary} \label{211}
Suppose that $I$ and $J$ are two hyperideals of a commutative Krasner $(m,n)$-hyperring $H$ with $J \subseteq I$. If $I$ is an $n$-ary $\delta$-$N$-hyperideal of $H$, then $I/J$ is an $n$-ary $\delta_q$-$J$-hyperideal of $H/J$. 
\end{corollary}
\begin{proof}
The claim follows by using Theorem \ref{210} (2) and by a similar argument to that of \ref{113}. 
\end{proof}
The next theorem shows that for an $n$-ary $\delta$-$N$-hyperideal $I$ of $H$ if $\delta (\sqrt{I}^{(m,n)})$ contains $\sqrt{\delta(I)}^{(m,n)}$, then $\sqrt{I}^{(m,n)}$ is an $n$-ary $\delta$-$N$-hyperideal of $H$. 
\begin{theorem}\label{212}
Suppose that  $I$ is an $n$-ary $\delta$-$N$-hyperideal of a commutative Krasner $(m,n)$-hyperring $H$ with $\sqrt{\delta(I)}^{(m,n)} \subseteq \delta (\sqrt{I}^{(m,n)})$. Then $\sqrt{I}^{(m,n)}$ is an $n$-ary $\delta$-$N$-hyperideal of $H$. 
\end{theorem}
\begin{proof}
Assume that  $g(x_1^n)\in \sqrt{I}^{(m,n)}$ for  $x_1^n \in H$ with $x_i \notin \sqrt{0}^{(m,n)}$ for some $1 \leq i \leq n$. Then there exists $t \in \mathbb{N}$ such that if $t \leq n$ implies that  $g(g(x_1^n)^{(t)},1_H^{(n-t)}) \in I$. Then we get the result that

  $g(x_i^{(t)},g(x_1^{i-1},1_H,x_{i+1}^n)^{(t)},1_H^{(n-2t)})=
g(g(x_i^{(t)},1_H^{(n-t)}),g(x_1^{i-1},1_H,x_{i+1}^n)^{(t)},$

$\hspace{0.5cm}1_H^{(n-t-1)}) \subseteq I$.\\
Since  $g(x_i^{(t)},1^{(n-t)}) \notin \sqrt{0}^{(m,n)}$ and $I$ is an $n$-ary $\delta$-$N$-hyperideal of $H$,   then we have $g(g(x_1^{i-1},1_H,x_{i+1}^n)^{(t)},1_H^{(n-t)}) \in \delta(I)$.
 It follows that  $g(x_1^{i-1},1_R,x_{i+1}^n) \in \sqrt{\delta(I)}^{(m,n)}$. Since $\sqrt{\delta(I)}^{(m,n)} \subseteq \delta (\sqrt{I}^{(m,n)})$, then $g(x_1^{i-1},1_H,x_{i+1}^n) \in \delta(\sqrt{I}^{(m,n)})$, as needed. A similar argument will show that if  $t=l(n-1)+1$, then   $\sqrt{I}^{(m,n)}$ is an $n$-ary $\delta$-$N$-hyperideal of $H$.
\end{proof}
Recall from \cite{sorc1} that a non-empty subset $S$ of a Krasner $(m,n)$-hyperring $H$ is said to be an $n$-ary multiplicative, if $g(t_1^n) \in S$ for $t_1,...,t_n \in S$.
The notion of Krasner $(m,n)$-hyperring of fractions was introduced in \cite{mah5}. 
Assume that $\delta$ is a hyperideal expansion of a commutative Krasner $(m,n)$-hyperring $H$ and $S$ is an $n$-ary multiplicative subset of $H$ with $1 \in S$. Then $\delta_S$ is a hyperideal expansion of $S^{-1}H$ with $\delta_S(S^{-1}I)=S^{-1}(\delta(I))$.
\begin{theorem} \label{113}
Suppose that  $S$ is an $n$-ary multiplicative subset of a commutative Krasner $(m,n)$-hyperring $H$ with $1 \in S$. If $I$ is an $n$-ary $\delta$-$N$-hyperideal of $H$ and $I \cap S=\varnothing$, then $S^{-1}I$ is an $n$-ary $\delta_S$-$N$-hyperideal of $S^{-1}H$.
\end{theorem}
\begin{proof}
Let $I$ be an $n$-ary $\delta$-$N$-hyperideal of $H$. Assume that $G(\frac{x_1}{s_1},...,\frac{x_n}{s_n}) \in S^{-1}I$ for $\frac{x_1}{s_1},...,\frac{x_n}{s_n} \in S^{-1}H$ such that $\frac{x_i}{s_i} \notin \sqrt{0_{S^{-1}H}}^{(m,n)}$ for some $1 \leq i \leq n$. 
Hence $\frac{g(x_1^n)}{g(s_1^n)} \in S^{-1}I$. This means that there  exists $t \in S$ such that $g(t,g(x_1^n),1_H^{(n-2)}) \in I$ and then $g(x_i,g(x_1^{i-1},t,x_{i+1}^n),1_H^{(n-2)}) \in I$.
 Since $x_i \notin \sqrt{0}^{(m,n)}$ and $I$ is an $n$-ary $\delta$-$N$-hyperideal of $H$, then we get the result that $g(x_1^{i-1},t,x_{i+1}^n) \in \delta(I)$. 
Therefore $G(\frac{x_1}{s_1},\cdots,\frac{x_{i-1}}{s_{i-1}},\frac{1_H}{1_H},\frac{x_{i+1}}{s_{i+1}},\cdots,\frac{x_{n}}{s_{n}})=\frac{g(x_1^{i-1},1_H,x_{i+1}^n)}{g(s_1^{i-1},1_H,s_{i+1}^n)}=\frac{g(x_1^{i-1},t,x_{i+1}^n)}{g(s_1^{i-1},t,s_{i+1}^n)} \in S^{-1}(\delta(I))=\delta_S(S^{-1}(I))$. Consequently,  $S^{-1}I$ is an $n$-ary $\delta_S$-$N$-hyperideal of $S^{-1}H$.
\end{proof}
\section{$(k,n)$-absorbing $\delta$-$N$-hyperideals}
In this section, we extend the concept of  $\delta$-$N$-hyperideals  to  the notion of $(k,n)$-absorbing $\delta$-$N$-hyperideals.
\begin{definition}
Given a  hyperideal expansion $\delta$, a proper hyperideal $I$ of a commutative Krasner $(m,n)$-hyperring $H$ is called $(k,n)$-absorbing $N$-hyperideal if  $g(a_1^{kn-k+1}) \in I$ for $a_1^{kn-k+1} \in H$ implies that $g(a_1^{(k-1)n-k+2}) \in \sqrt{0}^{(m,n)}$ or a $g$-product of $(k-1)n-k+2$ of $a_i^,$ s except $g(a_1^{(k-1)n-k+2})$ is in $\delta(I)$.
\end{definition}
\begin{example} 
If we continue with Example \ref{classes} and use its notation, then $I=\{0,3,6\}$ is a $(2,2)$-absorbing $\delta_1$-$N$-hyperring   of $\mathbb{Z}_{12}/\mathbb{Z}_{12}^\star$.
\end{example}
\begin{theorem}
Let  $I$ be a $(k,n)$-absorbing $N$-hyperideal of a commutative Krasner $(m,n)$-hyperring $H$. Then $\sqrt{I}^{(m,n)}$ is a $(k,n)$-absorbing $\delta$-$N$-hyperideal.
\end{theorem}
\begin{proof}
Assume that  $g(a_1^{kn-k+1}) \in \sqrt{I}^{(m,n)}$ for $a_1^{kn-k+1} \in H$. We presume none of the $g$-products of $(k-1)n-k+2$ of the $a_i^,$s other than $g(a_1^{(k-1)n-k+2})$ are in $\delta(\sqrt{I}^{(m,n)})$. Since $g(a_1^{kn-k+1}) \in \sqrt{I}^{(m,n)}$, then for some $t \in \mathbb{N}$ we have for $t \leq n$, $g(g(a_1^{kn-k+1})^{(t)},1_H^{(n-t)}) \in I$ or for $t>n$ with $t=l(n-1)+1$, $g_{(l)}(g(a_1^{kn-k+1})^{(t)}) \in I$. In the first possibilty, since all $g$-products of the $a_i^,$s other than $g(a_1^{(k-1)n-k+2})$ are not in $\delta(\sqrt{I}^{(m,n)})$, then they are not in $I$. Since $I$ is a $(k,n)$-absorbing $J$-hyperideal of $H$, then we have $g(g(x_1^{(k-1)n-k+2)})^{(l)},1_H^{(n-t)}) \in \sqrt{0}^{(m,n)}$ which means $g(x_1^{(k-1)n-k+2}) \in \sqrt{0}^{(m,n)}$. In the second possibilty, the claim follows by using a similar argument.
\end{proof}
\begin{theorem}
Suppose that $I$ is a hyperideal of a commutative Krasner $(m,n)$-hyperring $H$ such that $\delta(I)$ is a $(2,n)$-absorbing $N$-hyperideal. Then $I$ is a $(3,n)$-absorbing $\delta$-$N$-hyperideal of $H$.
\end{theorem}
\begin{proof}
Let $g(a_1^{3n-2}) \in I$  for $a_1^{3n-2} \in H$ but $g(a_1^{2n-1}) \notin \sqrt{0}^{(m,n)}$. This means that  $g(g(a_1,a_{2n}^{3n-2}),a_2^{2n-1}) \in I \subseteq \delta(I)$. Since $\delta(I)$ is a $(2,n)$-absorbing $N$-hyperideal of $H$ and $g(a_2^{2n-1}) \notin \sqrt{0}^{(m,n)}$, then we get the result that $g(a_1^n,a_{2n}^{3n-2}) \in \delta(I)$ or $g(a_1,a_{n+1}^{2n-1},a_{2n}^{3n-2}) \in \delta(I)$. Consequently,  $I$ is a $(3,n)$-absorbing $\delta$-$N$-hyperideal of $H$.
\end{proof}
\begin{theorem}
Assume that $I$ is a hyperideal of a commutative Krasner $(m,n)$-hyperring $H$ such that $\delta(I)$ is a $(k+1,n)$-absorbing $\delta$-$N$-hyperideal of $H$. Then $I$ is a $(k+1,n)$-absorbing $\delta$-$N$-hyperideal of $H$.
\end{theorem}
\begin{proof}
Suppose that  $g(a_1^{(k+1)n-(k+1)+1})
\in I$ for $a_1^{(k+1)n-(k+1)+1} \in H$ such that $g(a_1^{kn-k+1}) \notin \sqrt{0}^{(m,n)}$. So $g(a_1^{(k+1)n-(k+1)+1})=g(a_1^{kn-k},g(a_{kn-k+1}^{(k+1)n-(k+1)+1})) \in I \subseteq \delta(I)$.
Since  $\delta(I)$ is a $(k+1,n)$-absorbing $\delta$-$N$-hyperideal and $g(a_1^{kn-k+1}) \notin \sqrt{0}^{(m,n)}$, we get the result that $g(a_1^{i-1},a_{i+1}^{kn-k},g(a_{kn-k+1}^{(k+1)n-(k+1)+1})) \in \delta(I)$ for $1 \leq i \leq n$.  Thus $I$ is a $(k+1,n)$-absorbing $\delta$-$N$-hyperideal of $H$.
\end{proof}
\begin{theorem}
Let $I$ be a $\delta$-$N$-hyperideal of a commutative Krasner $(m,n)$-hyperring $H$. Then $I$ is a $(2,n)$-absorbing $\delta$-$N$-hyperideal of $H$.
\end{theorem}
\begin{proof}
Let  $g(a_1^{2n-1}) \in I$ for $a_1^{2n-1} \in H$. Since $I$ is a $\delta$-$N$-hyperideal of $H$,  we get the result that $g(a_1^n) \in \sqrt{0}^{(m,n)}$ or $g(a_{n+1}^{2n-1}) \in \delta (I)$. Therefore we have  $g(a_i,x_{n+1}^{2n-1}) \in \delta (I)$, for $1 \leq i \leq n$, as  $\delta(I)$ is a hyperideal of $H$. Thus $I$ is $(2,n)$-absorbing $\delta$-primary.
\end{proof}
Next, we determine all integers $k>n$.
\begin{theorem}
Let  $I$ is a $(k,n)$-absorbing $\delta$-$N$-hyperideal of a commutative Krasner $(m,n)$-hyperring $H$. Then $I$ is $(s,n)$-absorbing $\delta$-$N$-hyperideal for $s>n$. 
\end{theorem}
\begin{proof}
Let  $g(a_1^{(k+1)n-(k+1)+1}) \in I$ for $a_1^{(k+1)n-(k+1)+1} \in H$. Put $g(a_1^{n+2})=a$. Since $I$  is $(k,n)$-absorbing $\delta$-$N$-hyperideal, then we obtain $g(a,\cdots,a_{(k+1)n-(k+1)+1}) \in \sqrt{0}^{(m,n)}$ or a $g$-product of $kn-k+1$ of the $a_i^,$s except $g(a,\cdots,a_{(k+1)n-(k+1)+1})$ is in $\delta(I)$. This implies that  $g(a_i,a_{n+3}^{(k+1)n-(k+1)+1}) \in \delta(I)$ for all $1 \leq i \leq n+2$ which implies $I$ is a $(k+1,n)$-absorbing $\delta$-$N$-hyperideal. Consequently, $I$  is an $(s,n)$-absorbing $\delta$-$N$-hyperideal for $s>n$.
\end{proof}
\section{$n$-ary $S$-$N$-hyperideals} 
\begin{definition}
Assume that $S$ is  an $n$-ary multiplicative subset of a commutative Krasner $(m,n)$-hyperring $H$. A  hyperideal $I$ of $R$ with $I \cap S=\varnothing$ is said to be an $n$-ary $S$-$N$-hyperideal if there exists an $s \in S$ such that for all $x_1^n \in H$ if $g(x_1^n) \in I$ with $g(s,x_i,1_H^{(n-2)}) \notin \sqrt{0}^{(m,n)}$ for some $1 \leq i \leq n$, then $g(x_1^{i-1},s,x_{i+1}^n) \in I$. This element $s$ in $ S$ is called an $S$-element of $I$.
\end{definition}
\begin{example}
Let $H=\{0,1,a\}$. Consider commutative Krasner $(3,3)$-hyperring $(H,f,g)$ that 3-ary  $g$ is defined as $g(1,1,1)=1$,  $g(1,1,a)=g(1,a,a)=g(a,a,a)=a$ and for all $x,y \in H$, $g(0,x,y)=0$   and 3-ary hyperoeration $f$ is defined as 
$f(0,0,0)=0$, $f(0,0,1)=1$,  $f(0,1,1)=1$, $f(1,1,1)=1$, $f(1,1,a)=H$, $f(0,1,a)=H$, $f(0,0,a)=a$, $f(0,a,a)=a$, $f(1,a,a)=H$, $f(a,a,a)=a$.
In the Krasner $(3,3)$-hyperring, $I=\{0,a\}$ is an 3-ary $S$-$N$-hyperideal of $H$ such that 3-ary multiplicative subset $S$ is $\{1,2 \}$. 
\end{example}
Now we give a charactrization of an $n$-ary $S$-$N$-hyperideal.
\begin{theorem} \label{31} 
Assume that $S$ is an $n$-ary multiplicative subset of a commutative Krasner $(m,n)$-hyperring $H$ and $I$ is a hyperideal of $H$ disjoint with $S$. Then $I$ is an $n$-ary $S$-$N$-hyperideal of $H$ if and only if there exists $s \in S$, for all hyperideals $I_1^n$ of $H$, if $g(I_1^n) \subseteq I$, then $g(s,I_i,1^{(n-1)}) \subseteq \sqrt{0}^{(m,n)}$ for some $1 \leq i \leq n$ or $g(I_1^{i-1},s,I_{i+1}^n) \subseteq I$.
\end{theorem}
\begin{proof}
$\Longrightarrow$ Let $I$ be an $n$-ary $S$-$N$-hyperideal of $H$. Suppose that  $g(I_1^n) \in I$ for some hyperideals $I_1^n$ of $H$ such that $g(s,I_i,1_H^{(n-2)}) \nsubseteq \sqrt{0}^{(m,n)}$ and $g(I_1^{i-1},s,I_{i+1}^n) \nsubseteq I$ for all $s \in S$. Then there exists $a_i \in I_i$ for each $1 \leq i \leq n$ such that $g(a_1^n) \in I$ but $g(s,a_i,1_H^{(n-2)}) \notin \sqrt{0}^{(m,n)}$ and $g(a_1^{i-1},s,a_{i+1}^n) \notin I$, a contradiction. \\
$\Longleftarrow$ Suppose that $g(x_1^n) \in I$ for some $x_1^n \in H$. Then $ g(\langle x_1 \rangle, \cdots, \langle x_n \rangle) \subseteq I$. By the assumption, we are done.
\end{proof} 
\begin{theorem}
Let $S$ be an $n$-ary multiplicative subset of a commutative Krasner $(m,n)$-hyperring $H$. If $I_1^n$ are some $n$-ary $S$-$N$-hyperideals of $H$, then $\bigcap_{t=1}^nI_t$ is an $n$-ary $S$-$N$-hyperideal of $H$.
\end{theorem}
\begin{proof}
Let $I_1^n$ be $n$-ary $S$-$N$-hyperideals of  $H$. Suppose that for each $1 \leq t \leq n$,  there exists $s_t \in S$ such that if $g(a_1^n) \in I_t$ for some $a_1^n \in H$, then $g(s_t,a_i,1_H^{(n-2)}) \in \sqrt{0}^{(m,n)}$ or $g(a_1^{i-1},s_t,a_{i+1}^n) \in I_t$. Now, assume that $g(a_1^n) \in \bigcap_{t=1}^n I_t$ for some $a_1^n \in H$. This means that $g(a_1^n) \in  I_t$ for each $1 \leq t \leq n$. Put $\Pi_{t =1}^n s_t \in S$. Then we get the result thet $g(s,a_i,1_H^{(n-2)}) \in \sqrt{0}^{(m,n)}$ or $g(a_1^{i-1},s,a_{i+1}^n) \in \bigcap_{t=1}^n I_t$, as claimed.
\end{proof}
\begin{theorem} \label{32} 
Assume that $S$ is an $n$-ary multiplicative subset of a commutative Krasner $(m,n)$-hyperring $H$ and $I$ is a hyperideal of $H$ such that $I \cap S =\varnothing$. If $E_s=\{y \in H \ \vert \ g(y,s,1_H^{(n-2)}) \in I\}$ is an $n$-ary $N$-hyperideal of  $H$ for some $s \in S$, then $I$ is an $n$-ary $S$-$N$-hyperideal of $H$.
\end{theorem}
\begin{proof}
Let $E_s$ be an $n$-ary $N$-hyperideal of  $H$ for some $s \in S$. Suppose that $g(x_1^n) \in I$ for $x_1^n \in H$ such that $g(s,x_i,1_H^{(n-2)}) \notin \sqrt{0}^{(m,n)}$ for some $1 \leq i \leq n$. Therefore $g(g(x_1^n),s,1_H^{(n-2)})) \in I$ and so $g(x_1^n) \in E_s$. Since $E_s$ is an $n$-ary $N$-hyperideal of  $H$ for some $s \in S$ and $x_i \notin \sqrt{0}^{(m,n)}$, we get the result that $g(x_1^{i-1},1_H,x_{i+1}^n) \in E_s$ which means $g(g(x_1^{i-1},1_H,x_{i+1}^n),s,1_H^{(n-2)})=g(x_1^{i-1},s,x_{i+1}^n) \in I$. Consequently, $I$ is an $n$-ary $S$-$N$-hyperideal of $H$.
\end{proof}
In the following theorem, we determine a condition on $I$ when the converse holds.
\begin{theorem} \label{33} 
Assume that $S$ is an $n$-ary multiplicative subset of a commutative Krasner $(m,n)$-hyperring $H$ and $I$ is a hyperideal of $H$ such that $I \cap S =\varnothing$. If $I$ is an $n$-ary $S$-$N$-hyperideal of $H$ and $F_s=\{y \in H \ \vert \ g(y,s,1_H^{(n-2)}) \in \sqrt{0}^{(m,n)}\}$ is an $n$-ary $N$-hyperideal for an $S$-element $s \in S$ of $I$, then $E_s=\{y \in H \ \vert \ g(y,s,1_H^{(n-2)}) \in I\}$ is an $n$-ary $N$-hyperideal of  $H$.
\end{theorem}
\begin{proof}
Let $g(x_1^n) \in E_s$ for $x_1^n \in H$. Then we have $g(g(x_1^n),s,1_H^{(n-2)}) \in I$ and so $g(x_i,g(x_1^{i-1},s,x_{i+1}^n),1_H^{(n-2)}) \in I$. Since $I$ is an $n$-ary $S$-$N$-hyperideal of $H$, we get the result that either $g(s,x_i,1_H^{(n-2)}) \in \sqrt{0}^{(m,n)}$ or $g(g(x_1^{i-1},s,x_{i+1}^n),s,1_H^{(n-2)})=g(g(x_1^{i-1},1_H,x_{i+1}^n),g(s^{(2)},1_H^{(n-2)}),1_H^{(n-2)}) \in I$. In the former case, by Theorem \ref{11} we conclude that $F_s=\sqrt{0}^{(m,n)}$, as $F_s$ is an $n$-ary $N$-hyperideal. Therefore $x_i \in \sqrt{0}^{(m,n)}$. In the second case, suppose that $g(x_1^{i-1},s,x_{i+1}^n) \notin I$. Then we obtain $g(g(s^{(2},1_H^{(n-2)}),s,1_H^{(n-2)})=g(s^{(3)},1_H^{(n-3)}) \in \sqrt{0}^{(m,n)}$. It implies that $s \in \sqrt{0}^{(m,n)}$, a contradiction. Then we conclude that $g(x_1^{i-1},s,x_{i+1}^n) \in I$ which means $g(x_1^{i-1},1_H,x_{i+1}^n) \in E_s$ which implies $E_s$ is an $n$-ary $N$-hyperideal of  $H$.
\end{proof}
\begin{theorem}
Let $S \subseteq S^{\prime}$ be two $n$-ary multiplicative subsets of a commutative Krasner $(m,n)$-hyperring $H$ and $I$ be an $n$-ary $S^{\prime}$-$N$-hyperideal of $H$. If for each $s \in S^{\prime}$, there is an element $s^{\prime} \in S^{\prime}$ with $g(s,s^{\prime},1_H^{(n-2)}) \in S$, then $I$ is an $n$-ary $S$-$N$-hyperideal of $H$.
\end{theorem}
\begin{proof}
Let $g(a_1^n) \in I$. Since $I$ is an $n$-ary $S^{\prime}$-$N$-hyperideal of $H$,  we have either $g(s,a_i,1_H^{(n-2)}) \in \sqrt{0}^{(m,n)}$  or $g(a_1^{i-1},s,a_{i+1}^n) \in I$ for a $S^{\prime}$-element $s \in S^{\prime}$ of $I$. By the assumption, there exists $s^{\prime} \in S^{\prime}$ such that $s^{\prime \prime}=g(s,s^{\prime},1_H^{(n-2)}) \in S$. From $g(s,a_i,1_H^{(n-2)}) \in \sqrt{0}^{(m,n)}$ , it follows that $g(s^{\prime},g(s,a_i,1_H^{(n-2)}),1_H^{(n-2)}) \in \sqrt{0}^{(m,n)}$. Also, from $g(a_1^{i-1},s,a_{i+1}^n) \in I$ it follows that $g(s^{\prime},g(a_1^{i-1},s,a_{i+1}^n),1_H^{(n-2)}) \in I$. Consequently, we get the result that $g(s^{\prime \prime},a_i,1_H^{(n-2)}) \in \sqrt{0}^{(m,n)}$ or $g(a_1^{i-1},s^{\prime \prime},a_{i+1}^n) \in I$, as needed.
\end{proof}
\begin{theorem} \label{34}
Let $S \subseteq S^{\prime}$ be two $n$-ary multiplicative subsets of a commutative Krasner $(m,n)$-hyperring $H$  such that $1_H \in S$ and $I$ be a hyperideal of $H$ with $I \cap S^{\prime}=\varnothing$ . If $I$ is an $n$-ary $S$-$N$-hyperideal of $H$, then ${S^{\prime}}^{-1}I$ is an $n$-ary ${S^{\prime}}^{-1}S$-$N$-hyperideal of ${S^{\prime}}^{-1}H$ and ${S^{\prime}}^{-1}I \cap H=E_s$ where $E_s=\{y \in H \ \vert \ g(y,s,1^{(n-2)}) \in I\}$ and $ s $ is an $S$-element  of $I$.
\end{theorem}
\begin{proof}
Let $I$ be an $n$-ary $S$-$N$-hyperideal of $H$. It is easy to see that ${S^{\prime}}^{-1}S \cap {S^{\prime}}^{-1}I = \varnothing$ . Assume that $\frac{s}{1_H} \in {S^{\prime}}^{-1}S$ for some $S$-element $ s $ of $I$. Let $G(\frac{a_1}{s_1},\cdots,\frac{a_n}{s_n}) \in {S^{\prime}}^{-1}I$ for $a_1^n \in H$ and $s_1^n \in S^{\prime}$ such that $G(\frac{s}{1_H},\frac{a_i}{s_i},\frac{1_H}{1_H}^{(n-2)}) \notin \sqrt{0_{{S^{\prime}}^{-1}H}}^{(m,n)}$ for some $1 \leq i \leq n$. So $\frac{g(a_1^n)}{g(s_1^n)}\in {S^{\prime}}^{-1}I$
which follows  there  exists $t \in S^{\prime}$ such that $g(t,g(a_1^n),1_H^{(n-2)}) \in I$ and then $g(a_i,g(a_1^{i-1},t,a_{i+1}^n),1_H^{(n-2)}) \in I$. Since $I$ is an $n$-ary $S$-$N$-hyperideal of $H$ and
$g(s,a_i,1_H^{(n-2)}) \notin \sqrt{0_H}^{(m,n)}$, we get the result that $g(s,g(a_1^{i-1},t,a_{i+1}^n),1_H^{(n-2)}) \in I$ which means $G(\frac{a_1}{s_1},\cdots,\frac{a_{i-1}}{s_{i-1}},\frac{s}{1_H},\frac{a_{i+1}}{s_{i+1}},\cdots,\frac{a_n}{s_n})=\frac{g(a_1^{i-1},s,a_{i+1}^n)}{g(s_1^{i-1},1_H,s_{i+1}^n)}=\frac{g(s,g(a_1^{i-1},t,a_{i+1}^n),1^{(n-2)})}{g(s_1^{i-1},t,s_{i+1}^n)} \in {S^{\prime}}^{-1}I$. Thus ${S^{\prime}}^{-1}I$ is an $n$-ary ${S^{\prime}}^{-1}S$-$N$-hyperideal of ${S^{\prime}}^{-1}H$. For the second assertion, suppose that $x \in {S^{\prime}}^{-1}I \cap H$. Then there exists $a \in I$ such that $\frac{x}{1_H}=\frac{a}{t}$ for some $t \in S^{\prime}$. Therefore there exists $u \in S^{\prime}$ such that $g(u,a,1_H^{(n-2)}) \in I$. Since $I$ is an $n$-ary $S$-$N$-hyperideal of $H$, then there exists $s \in S \subseteq S^{\prime}$ such that we have $g(s,u,1_H^{(n-2)}) \in \sqrt{0_H}^{(m,n)}$ or $g(s,a,1_H^{(n-2)}) \in I$. In the former case, we
have a contradiction since  $S^{\prime} \cap \sqrt{0_H}^{(m,n)}=\varnothing$. Hence $g(s,a,1_H^{(n-2)}) \in I$ which implies $a \in E_s$ which means $ {S^{\prime}}^{-1}I \cap H \subseteq E_s$. 
Since the inclusion $E_s \subseteq {S^{\prime}}^{-1}I \cap H$ holds, we get ${S^{\prime}}^{-1}I \cap H=E_s$.
\end{proof}
We have the following corollary of Theorem \ref{34}.
\begin{corollary}\label{35}
Let $S$ be an $n$-ary multiplicative subset of a commutative Krasner $(m,n)$-hyperring $H$  such that $1_H \in S$ and $I$ be a hyperideal of $H$ with $I \cap S=\varnothing$ . If $I$ is an $n$-ary $S$-$N$-hyperideal of $H$, then $S^{-1}I$ is an $n$-ary $N$-hyperideal of $S^{-1}H$ and $S^{-1}I \cap H=E_s$   where $E_s=\{y \in H \ \vert \ g(y,s,1^{(n-2)}) \in I\}$ and $ s $ is an $S$-element  of $I$.
\end{corollary}
\begin{proof}
Let $I$ be an $n$-ary $S$-$N$-hyperideal of $H$. Then, by Theorem \ref{34}, $S^{-1}I$ is an $n$-ary $S^{-1}S$-N-hyperideal of $S^{-1}H$. Suppose that $G(\frac{a_1}{s_1},\cdots,\frac{a_n}{s_n}) \in S^{-1}I$ for $a_1^n \in H$ and $s_1^n \in S$. There there exist an $S^{-1}S$-element $\frac{u}{v}$ of $S^{-1}I$ such that $G(\frac{u}{v},\frac{a_i}{s_i},\frac{1_H}{1_H}^{(n-2)}) \in \sqrt{0_{S^{-1}H}}^{(m,n)}$  or $G(\frac{a_1}{s_1},\cdots,\frac{a_{i-1}}{s_{i-1}},\frac{u}{v},\frac{a_{i+1}}{s_{i+1}},\cdots,\frac{a_n}{s_n}) \in  S^{-1}I$ for some $1 \leq i \leq n$. Then we get the result that $S^{-1}I$ is an $n$-ary $N$-hyperideal of $S^{-1}H$, as $\frac{u}{v}$ is invertible. For the second assertion, we use a similar argument to that of Theorem \ref{34}.
\end{proof}
From Theorem \ref{34} and Corollary \ref{35}, we can conclude the
following.
\begin{theorem}
Let $S$ be an $n$-ary multiplicative subset of a commutative Krasner $(m,n)$-hyperring $H$  such that $1_H \in S$ and $I$ be a hyperideal of $H$ with $I \cap S=\varnothing$. Then $I$ is an $n$-ary $S$-$N$-hyperideal of $H$ if and only if  $S^{-1}I$ is an $n$-ary $N$-hyperideal of $S^{-1}H$ and $S^{-1}I \cap H=E_s$   and  $S^{-1}\sqrt{0}^{(m,n)} \cap H=F_{s^{\prime}}$ where $E_s=\{y \in H \ \vert \ g(y,s,1^{(n-2)}) \in I\}$ and $F_{s^\prime}=\{z \in H \ \vert \ g(z,s^{\prime},1^{(n-2)}) \in \sqrt{0}^{(m,n)}\}$ for some $s,s^{\prime} \in S$.
\end{theorem}
\begin{proof}
$\Longrightarrow$ Let $I$ be an $n$-ary $S$-$N$-hyperideal of $H$. Then, by Corollary \ref{35} we conclude that $S^{-1}I$ is an $n$-ary $N$-hyperideal of $S^{-1}H$. The rest of the claim follows by a similar argument to that of Theorem \ref{34}.\\
$\Longleftarrow$ Let $S^{-1}I$ be an $n$-ary $N$-hyperideal of $S^{-1}H$, $S^{-1}I \cap H=E_s$   and  $S^{-1}\sqrt{0}^{(m,n)} \cap H=F_{s^{\prime}}$ where $E_s=\{y \in H \ \vert \ g(y,s,1_H^{(n-2)}) \in I\}$ and $F_{s^\prime}=\{z \in H \ \vert \ g(z,s^{\prime},1_H^{(n-2)}) \in \sqrt{0}^{(m,n)}\}$ for some $s,s^{\prime} \in S$. Put $g(s,s^{\prime},1_H^{(n-2)})=s^{\prime \prime}$. Assume that $g(a_1^n) \in I$ for $a_1^n \in H$. This implies that $G(\frac{a_1}{1_H},\cdots,\frac{a_n}{1_H}) \in S^{-1}I$. Then we get the result that either $\frac{a_i}{1_H} \in \sqrt{S^{-1}0}^{(m,n)}=S^{-1}\sqrt{0}^{(m,n)}$ or $G(\frac{a_1}{1_H},\cdots,\frac{a_{i-1}}{1_H},\frac{1_H}{1_H},\frac{a_{i+1}}{1_H},\cdots,\frac{a_n}{1_H}) \in S^{-1}I$ for some $1 \leq i \leq n$. In the first case, we get $g(t,a_i,1_H^{(n-2)}) \in \sqrt{0}^{(m,n)}$ for some $t \in S$ which implies $a_i=\frac{g(t,a_i,1_H^{(n-2)})}{g(t,1_H^{(n-1)})} \in S^{-1}\sqrt{0}^{(m,n)} \cap H=F_{s^{\prime}}$. This means that $g(s^{\prime},a_i,1_H^{(n-2)}) \in \sqrt{0}^{(m,n)}$. Then $g(s^{\prime \prime},a_i,1_H^{(n-2)})=g(g(s,s^{\prime},1_H^{(n-2)}),a_i,1_H^{(n-2)}) \in \sqrt{0}^{(m,n)}$. In the second case, we get $g(a_1^{i-1},t^{\prime},a_{i+1}^n) \in I$ for some $t^{\prime} \in S$. This implies that $g(a_1^{i-1},1_H,a_{i+1}^n)=\frac{g(a_1^{i-1},t^{\prime},a_{i+1}^n)}{g(t^{\prime},1_H^{(n-1)})} \in S^{-1}I \cap H=E_s$. Then we conclude that  $g(a_1^{i-1},s^{\prime \prime},a_{i+1}^n)=g(a_1^{i-1},g(s,s^{\prime},1_H^{(n-2)}),a_{i+1}^n) \in I$. Consequently, $I$ is an $n$-ary $S$-$N$-hyperideal of $H$.
\end{proof}
\begin{theorem} \label{36}
Let $S_1$ and $S_2$ be  $n$-ary multiplicative subsets of commutative Krasner $(m,n)$-hyperrings $(H_1, f_1, g_1)$ and $(H_2, f_2, g_2)$, respectively, and let  $1_{H_1}$ and $1_{H_2}$ be scalar identitis of $H_1$ and $H_2$, respectively.  Then the following statements hold.
\begin{itemize} 
\item[\rm(1)]~ $I_1 \times H_2$ is an $n$-ary $S$-$N$-hyperideal of $H_1 \times H_2$ where $S=S_1 \times S_2$ if and only if $I_1$ is an $n$-ary $S_1$-$N$-hyperideal of $H_1$ and $S_2 \cap \sqrt{0_{H_2}} \neq \varnothing$.
\item[\rm(2)]~ $H_1 \times I_2$ is an $n$-ary $S$-$N$-hyperideal of $H_1 \times H_2$ where $S=S_1 \times S_2$ if and only if $I_2$ is an $n$-ary $S_2$-$N$-hyperideal of $H_2$ and $S_1 \cap \sqrt{0_{H_1}} \neq \varnothing$.
\end{itemize}
\end{theorem}
\begin{proof}
$\Longrightarrow$ Let $I_1 \times H_2$ be an $n$-ary $S$-$N$-hyperideal of $H_1 \times H_2$ where $S=S_1 \times S_2$ and let $(s_1,s_2)$ be an $S$-element of $I_1 \times H_2$. Suppose that $g_1(a_1^n) \in I_1$ for some $a_1^n \in H_1$. So we consider the following cases.\\
Case {\bf 1}: Let $g_1(a_1^{i-1},s_1,a_{i+1}^n) \notin I_1$. Then we get $g_1 \times g_2((a_1,1_{H_2}),\cdots,(a_n,1_{H_2})) \in I_1 \times H_2$ but $g_1 \times g_2((a_1,1_{H_2}),\cdots,(a_{i-1},1_{H_2}),(s_1,s_2),(a_{i+1},1_{H_2}),\cdots(a_n,1_{H_2})) \notin I_1 \times H_2$. Then we get the result that $g_1 \times g_2((s_1,s_2),(a_i,1_{H_2}),(1_{H_1},1_{H_2})^{(n-2)}) \in \sqrt{0_{H_1 \times H_2}}=\sqrt{0_{H_1}} \times \sqrt{0_{H_2}}$, as $I_1 \times H_2$ is an $n$-ary $S$-$N$-hyperideal of $H_1 \times H_2$. Hence $g(s_1,a_i,1_{H_1}^{(n-2)}) \in \sqrt{0_{H_1}}$ and $s_2 \in S_2 \cap \sqrt{0_{H_2}}$, as needed.\\
Case {\bf 2}: In this case we suppose that  $g_1(a_1^{i-1},s_1,a_{i+1}^n) \in I_1$. Then we have $g_1 \times g_2((a_1,1_{H_2}),\cdots,(a_{i-1},1_{H_2}),(s_1,s_2),(a_{i+1},1_{H_2}),\cdots(a_n,1_{H_2})) \in I_1 \times H_2$. So we get the result that $g_1 \times g_2((a_1,1_{H_2}),\cdots,(a_{i-1},1_{H_2}),(s_1,s_2),(a_{i+1},1_{H_2}),\cdots(a_n,1_{H_2})) \in \sqrt{0_{H_1 \times H_2}}=\sqrt{0_{H_1}} \times \sqrt{0_{H_2}}$, since $g_1 \times g_2((s_1,s_2)^{(2)},(1_{H_1},1_{H_2})^{(n-2)}) \notin I_1 \times H_2$.\\
$\Longleftarrow$  Let $I_1$ be an $n$-ary $S_1$-$N$-hyperideal of $H_1$ and let $s_1$ be an $S_1$-element of $I_1$ and $s_2 \in S_2 \cap \sqrt{0_{H_2}}$. Now, assume that $g_1 \times g_2((a_1,b_1), \cdots, (a_n,b_n)) \in I_1 \times H_2$ for some $a_1^n \in H_1$ and $b_1^n \in H_2$. Then we have $g_1(a_1^n) \in I_1$. Hence we get either  $g_1(s_1,a_i,1_H^{(n-2)}) \in \sqrt{0_{H_1}}$ for some $1 \leq i \leq n$ or $g(a_1^{i-1},s_1,a_{i+1}^n) \in I_1$. Therefore $g_1 \times g_2((s_1,1_{H_2}),(a_i,b_i),(1_{H_1},1_{H_2})^{(n-2)}) \in \sqrt{0_{H_1}} \times \sqrt{0_{H_2}}$ or $g_1 \times g_2((a_1,b_1),\cdots,(a_{i-1},b_{i-1}),(s_1,s_2),(a_{i+1},b_{i+1}),\cdots,(a_n,b_n)) \in I_1 \times H_2$. So $(s_1,s_2)$ is an $S$-element of $I_1 \times H_2$.

(2) The proof is similar to (1).
\end{proof}
\begin{theorem} \label{37}
Let $S_1$ and $S_2$ be  $n$-ary multiplicative subsets of commutative Krasner $(m,n)$-hyperrings $(H_1, f_1, g_1)$ and $(H_2, f_2, g_2)$, respectively, and let  $I_1$ and $I_2$ be proper hyperideals of $H_1$ and $H_2$, respectively.  If one of the following cases  holds, 
\begin{itemize} 
\item[\rm(1)]~  $\sqrt{0_{H_2}}^{(m,n)} \cap S_2 \neq \varnothing$ and $I_1$ is an $n$-ary $S_1$-$N$-hyperideal of $H_1$.
\item[\rm(2)]~   $\sqrt{0_{H_1}}^{(m,n)} \cap S_1 \neq \varnothing$ and $I_2$ is an $n$-ary $S_2$-$N$-hyperideal of $H_2$.
\end{itemize}
then $I_1 \times I_2$ is an $n$-ary $S$-$N$-hyperideal of $H_1 \times H_2$ where $S=S_1 \times S_2$.
\end{theorem}
\begin{proof}
Let $\sqrt{0_{H_2}}^{(m,n)} \cap S_2 \neq \varnothing$ and $I_1$ be an $n$-ary $S_1$-$N$-hyperideal of $H_1$. So $0_{H_2} \in I_2 \cap S_2 \neq \varnothing$ and $I_1 \cap S_1 =\varnothing$. Suppose that $g_1 \times g_2((a_1,b_1),\cdots,(a_n,b_n)) \in I_1 \times I_2$ for some $a_1^n \in H_1$ and $b_1^n \in H_2$ and $s_1$ is an $S_1$-element of $I_1$. Hence we have $g_1(a_1^n) \in I_1$. Then either $g_1(s_1,a_i,1_{H_1}^{(n-2)}) \in \sqrt{0_{H_1}}^{(m,n)}$ for some $1 \leq i \leq n$ or $g_1(a_1^{i-1},s_1,a_{i+1}^n) \in I_1$. So  $g_1 \times g_2((s_1,0),(a_i,b_i),(1_{H_1},1_{H_2})^{(n-2)}) \in \sqrt{0_{H_1}}^{(m,n)} \times \sqrt{0_{H_2}}^{(m,n)}$ or $g_1 \times g_2((a_1,b_1),\cdots,(a_{i-1},b_{i-1}),(s_1,0),(a_{i+1},b_{i+1}),\cdots,(a_n,b_n)) \in I_1 \times I_2$. Thus $(s_1,0_{H_2})$ is an $S$-element of $I_1 \times I_2$.\\
Also, if Case 2 holds, then by using a similar argument, we can show that  $I_1 \times I_2$ is an $n$-ary $S$-$N$-hyperideal of $H_1 \times H_2$ where $S=S_1 \times S_2$.
\end{proof}


\end{document}